\documentclass[11pt]{article}

\usepackage[a4paper,left=3cm,right=3cm,top=3cm,bottom=3cm,bindingoffset=5mm]{geometry}

\usepackage[utf8]{inputenc}

\usepackage{amsmath, amsthm, amssymb}
\usepackage{color,hyperref}
\usepackage[all,cmtip]{xy}
\usepackage{enumerate}
\usepackage{color}

\newtheorem{mydef}{Definition}[section]
\newtheorem{thm}[mydef]{Theorem}
\newtheorem{prop}[mydef]{Proposition}
\newtheorem{lem}[mydef]{Lemma}
\newtheorem{conj}[mydef]{Conjecture}
\newtheorem{cor}[mydef]{Corollary}

\theoremstyle{remark}
\newtheorem{rem}[mydef]{Remark}
\newtheorem{ex}[mydef]{Example}
\newtheorem*{notations}{Notations}
\newtheorem*{ac}{Acknowledgments}

\title{$\ell$-Galois special subvarieties and the Mumford-Tate conjecture}
\author{Tobias Kreutz}
\date{\vspace{-5ex}}

\begin{document}

\maketitle

\begin{abstract}
We introduce $\ell$-Galois special subvarieties as an $\ell$-adic analog of the Hodge-theoretic notion of a special subvariety.
The Mumford-Tate conjecture predicts that both notions are equivalent.
We study some properties of these subvarieties and prove this equivalence for subvarieties satisfying a simple monodromy condition.
As applications, we show that the $\ell$-Galois exceptional locus is a countable union of algebraic subvarieties and, if the derived group of the generic Mumford-Tate group of a family is simple, its part of positive period dimension coincides with the Hodge locus of positive period dimension.
We use this to prove that for $n$ and $d$ sufficiently large, the absolute Mumford-Tate conjecture in degree $n$ holds on a dense open subset of the moduli space of smooth projective hypersurfaces of degree $d$ in $\mathbb{P}^{n+1}$, with the exception of hypersurfaces defined over number fields.
Finally, we show that the Mumford-Tate conjecture for abelian varieties is equivalent to a conjecture about the local structure of $\ell$-Galois special subvarieties in $\mathcal{A}_g$.
\end{abstract}

\renewcommand{\baselinestretch}{0.90}\normalsize
\tableofcontents
\renewcommand{\baselinestretch}{1.00}\normalsize

\section*{Introduction}
\addcontentsline{toc}{section}{Introduction}

The Mumford-Tate conjecture concerns the comparison of two additional structures on the cohomology of algebraic varieties.
Let $K \subset \mathbb{C}$ be a field finitely generated over $\mathbb{Q}$ and $X \to \mathrm{Spec \,\,} K$ a smooth projective algebraic variety.
\begin{enumerate}[(i)]
\item Betti cohomology $H^k_B(X, \mathbb{Q}):= H_{sing}^k(X(\mathbb{C}), \mathbb{Q})$ carries a $\mathbb{Q}$-Hodge structure.
\item
The $\ell$-adic étale cohomology $H^k_{\acute{e}t}(X_{\bar{K}}, \mathbb{Q}_{\ell})$ carries an action of the absolute Galois group of $K$, giving rise to a continuous Galois representation $$\rho_X: \mathrm{Gal}(\bar{K}/K) \to \mathrm{GL}(H_{\acute{e}t}^k(X_{\bar{K}}, \mathbb{Q}_{\ell})).$$
\end{enumerate}

The Mumford-Tate conjecture states that these additional structures convey essentially the same information. It is best formulated in terms of two groups naturally associated with the two structures.
We let $G_{X}$ denote the \emph{Mumford-Tate group} of $H^k_B(X, \mathbb{Q})$, a reductive group over $\mathbb{Q}$ which is defined as the Tannakian group of the subcategory tensorially generated by $H^k_B(X, \mathbb{Q})$ (and its dual) inside the Tannakian category of $\mathbb{Q}$-Hodge structures.
For the $\ell$-adic realization, we denote by $G_{\ell, X}$ the \emph{$\ell$-adic algebraic Galois group} of $X$, defined as the connected component of the identity of the Zariski closure of the image of $\rho_X$ in $\mathrm{GL}(H_{\acute{e}t}^k(X_{\bar{K}}, \mathbb{Q}_{\ell}))$.
\begin{conj}[Mumford-Tate conjecture]\label{mtconjecture}
For every smooth projective variety $X$ over $K$, we have the equality $G_{\ell,X} = G_{X} \otimes_{\mathbb{Q}} \mathbb{Q}_{\ell}$.
\end{conj}
There exist a number of partial results on the Mumford-Tate conjecture in the case where $X$ is an abelian variety or more generally, a variety whose motive is of abelian type. We refer to (\cite{Moonen}, 2.4 and 4.4) for a discussion of known cases of the conjecture.
In general however it is still open even for abelian varieties, let alone outside of the case of abelian motives.

Due to the very different nature of the groups $G_X$ and $G_{\ell,X}$, one often defines a third group $G_X^{AH}$ with the property that $G_X \subset G_X^{AH}$ and $G_{\ell,X} \subset G_X^{AH} \otimes_{\mathbb{Q}} \mathbb{Q}_{\ell}$. We call the group $G_X^{AH}$ the \emph{absolute Mumford-Tate group} of $X$, it is the motivic Galois group of the motive $\mathfrak{h}^k(X)$ in terms of motives for absolute Hodge cycles (cf. \cite{DM}, §6).

The Mumford-Tate conjecture then follows from the conjunction of the following two conjectures:

\begin{conj}[Hodge cycles are absolute Hodge]\label{HAH}
For every smooth projective variety $X$ over $\mathbb{C}$, we have the equality $G_{X} = G_{X}^{AH}$.
\end{conj}

\begin{conj}[absolute Mumford-Tate conjecture]\label{absMT}
For every smooth projective variety $X$ over $K$, we have the equality $G_{\ell,X} = G^{AH}_{X} \otimes_{\mathbb{Q}} \mathbb{Q}_{\ell}$.
\end{conj}

\subsection*{$\ell$-Galois special subvarieties}
Instead of working with a single variety $X$, it turns out to be beneficial to work with a family of varieties.
Let $f:\mathcal{X} \to S$ be a smooth projective morphism of smooth irreducible quasi-projective algebraic varieties defined over $K$.
For any integer $k \ge 0$ the $k$-th cohomology of this family together with the embedding $K \subset \mathbb{C}$ gives rise to a polarizable $\mathbb{Q}$-variation of Hodge structure
$(\mathbb{V}, \mathcal{V}, \nabla, F^{\bullet})$, where
$\mathbb{V} = R^kf_{\mathbb{C},*} \underline{\mathbb{Q}}_{\mathcal{X}^{an}}$ is a $\mathbb{Q}$-local system on $S^{an}$,
$(\mathcal{V} = R^kf_{*} \Omega^{\bullet}_{\mathcal{X}/S}, \nabla)$ the corresponding algebraic vector bundle with flat connection and $F^{\bullet}$ is the Hodge filtration of the family.
Similarly, the étale cohomology of the family gives rise to an étale $\mathbb{Q}_{\ell}$-local system
$ \mathbb{L}= R^kf_{*} \underline{\mathbb{Q}_{\ell}}_{\mathcal{X}_{\acute{e}t}}$ on $S$ defined over $K$.
After choosing a point $\bar{s} \in S(\bar{K})$, we denote by $\rho_{\ell}: \pi_1^{\acute{e}t}(S_K,\bar{s}) \to GL(\mathbb{L}_{\bar{s}})$ the associated arithmetic monodromy representation.
In fact, instead of the full cohomology, we will often only consider its primitive part.

Any closed irreducible complex subvariety $Z \subset S$ can be defined over a finitely generated extension $L$ of $K$, we choose a geometric point $\bar{z }\in Z(\bar{L})$.
We define the \emph{generic Mumford-Tate group} $G_Z$ as the subgroup of $GL(\mathbb{V}_{\bar{z}})$ fixing all generic Hodge tensors over $Z$.
Similarly, we define the \emph{generic absolute Mumford-Tate group} $G^{AH}_Z$ to be the subgroup of $GL(\mathbb{V}_{\bar{z}})$ fixing all generic absolute Hodge tensors over $Z$.
In addition, we define the \emph{generic $\ell$-adic algebraic Galois group} $G_{\ell,Z}$ to be the connected component of the identity of the Zariski closure of the image of the restricted arithmetic monodromy representation
$ \rho_{\ell, Z}: \pi_1^{\acute{e}t}(Z_L , \bar{z}) \to GL(\mathbb{L}_{\bar{z}})$. One checks that the definition of $G_{\ell,Z}$ does not depend on the choice of the field $L$.
Note that if $Z= \{z\}$ is a point and $X$ is the fiber of $f$ over $z$, we recover the three groups $G_{X}$, $G_X^{AH}$ and $G_{\ell,X}$ defined in the beginning.

Since every absolute Hodge tensor is at the same time a Hodge tensor and a Tate tensor, we have the inclusions $G_Z \subset G_Z^{AH}$
and $G_{\ell,Z} \subset G_Z^{AH}\otimes \mathbb{Q}_{\ell}$. As we shall see in Section \ref{usualMT}, Conjectures \ref{HAH} and \ref{absMT} imply that both inclusions should be equalities.

\begin{conj} \label{conjdel}
For any subvariety $Z \subset S$, we have $G_{ Z }^{AH} = G_Z$.
\end{conj}

\begin{conj}\label{conjtate}
For any subvariety $Z \subset S$, we have $G_{\ell, Z } = G_Z^{AH} \otimes_{\mathbb{Q}} \mathbb{Q}_{\ell}$.
\end{conj}

A great benefit of working in families is that we get natural algebraic subvarieties of the base $S$ which are defined by imposing the existence of certain generic Hodge cycles.

\begin{mydef}[\cite{KO}, Definition 1.2] \label{defspintro}
A closed irreducible complex subvariety $Z \subset S$ is called \emph{special} if it is maximal among the closed irreducible complex subvarieties $Y$ of $S$ whose generic Mumford-Tate group $G_Y$ equals $G_Z$.
\end{mydef}

The above Conjectures \ref{conjdel} and \ref{conjtate} suggest to define the following variants of special subvarieties:

\begin{mydef}[see Definitions \ref{definitionspecial} and \ref{defGalsp}]\label{introdefGalsp}
\textnormal{ }
\begin{enumerate}[(i)]
\item
A closed irreducible complex subvariety $Z \subset S$ is called \emph{absolutely special} if it is maximal among the closed irreducible complex subvarieties $Y$ of $S$ whose generic absolute Mumford-Tate group $G^{AH}_Y$ equals $G^{AH}_Z$.
\item
A closed irreducible complex subvariety $Z \subset S$ is called \emph{$\ell$-Galois special} if it is maximal among the closed irreducible complex subvarieties $Y$ of $S$ whose generic $\ell$-adic algebraic Galois group $G_{\ell,Y}$ equals $G_{\ell,Z}$.
\end{enumerate}
\end{mydef}

A variant of the notion of absolutely special subvarieties which only incorporates the de Rham and not the étale components of absolute Hodge cycles was introduced and studied in \cite{Abssppaper}.
However, in this paper we use the stronger notion of absolute Hodge cycles to ensure that every absolutely special subvariety is both special and $\ell$-Galois special (Proposition \ref{absspecial}).
Conjectures \ref{conjdel} and \ref{conjtate} predict that in fact all three notions are equivalent.
We are thus led to formulate the following conjectures.

\begin{conj}\label{everyspabssp}
Any special subvariety is absolutely special.
\end{conj}
\begin{conj}\label{everyGalspabssp}
Any $\ell$-Galois special subvariety is absolutely special.
\end{conj}

For general families very little is known concerning Conjectures \ref{conjdel} and \ref{conjtate}.
In this paper, we prove Conjectures \ref{everyspabssp} and \ref{everyGalspabssp} for subvarieties satisfying a simple monodromy condition introduced in \cite{KOU}.

We define the algebraic monodromy group $H_Z$ of a closed irreducible subvariety $Z \subset S$ to be the connected component of the identity of the Zariski closure of the image of the topological monodromy representation $\rho_Z: \pi_1(Z^{an},\bar{z}) \to GL(\mathbb{V}_{\bar{z}}) $ corresponding to the restriction of the local system $\mathbb{V}$ to $Z$.

\begin{mydef}[\cite{KOU}, Definition 1.10]
A closed irreducible subvariety $Z \subset S$ is called \emph{weakly non-factor} if it is not contained in a closed irreducible $Y \subset S$ such that $H_{Z}$ is a strict normal subgroup of $H_{Y}$.
\end{mydef}

\begin{thm}[see Theorem \ref{equiv}]\label{mainthm}
For a weakly non-factor subvariety $Z$, the following are equivalent:

\begin{enumerate}[(i)]
\item
the subvariety $Z$ is special;
\item
the subvariety $Z$ is $\ell$-Galois special for all primes $\ell$;
\item
the subvariety $Z$ is absolutely special.
\end{enumerate}

\end{thm}

The proof of Theorem \ref{mainthm} uses monodromy methods, and does not give information on the relation of the groups $G_Z$, $G_Z^{AH}$ and $G_{\ell,Z}$ to each other. In particular, the strategy of the proof cannot be extended to prove Conjectures \ref{conjdel} or \ref{conjtate}.

\subsection*{Exceptional loci}

Recall that a point $s \in S(\mathbb{C})$ is called \emph{Hodge generic} if $G_s = G_S$, and \emph{$\ell$-Galois generic} if $G_{\ell,s}=G_{\ell,S}$.
One often studies the complement of the set of Hodge resp. $\ell$-Galois generic points as the locus where we expect exceptional Hodge resp. Tate tensors to occur.

\begin{mydef}
We define the \emph{Hodge locus} $HL$ of $S$ to be the set of $s \in S(\mathbb{C})$ such that $G_s \subsetneq G_S$.
Similarly, we define the \emph{absolute Hodge locus} $AHL$ (respectively the \emph{$\ell$-Galois exceptional locus} $\ell-GalL$) to be the set of all $s \in S(\mathbb{C})$ such that $G_s^{AH} \subsetneq G_S^{AH}$ (respectively $G_{\ell, s} \subsetneq G_{\ell,S}$).
\end{mydef}

It is easy to see that the loci $HL$ (resp. $AHL$, $\ell-GalL$) are exactly the union of all strict special (resp. absolutely special, $\ell$-Galois special) subvarieties of $S$.
A fundamental result in Hodge theory due to Cattani-Deligne-Kaplan \cite{CDK} shows that there are only countably many special subvarieties of $S$, and thus the Hodge locus is a countable union of closed algebraic subvarieties, as predicted by the Hodge conjecture.

One easily checks that absolutely special subvarieties are defined over finite extensions of $K$ and the collection of absolutely special subvarieties is stable under $\mathrm{Gal}(\bar{K} /K)$.
Note however that this is unknown in general for special subvarieties, with notable exception the results in \cite{KOU}.
We will see in Corollary \ref{fodgalsp} that the same holds for $\ell$-Galois special subvarieties, thus proving the following $\ell$-adic analog of the theorem on the algebraicity of Hodge loci:

\begin{thm}[see Theorem \ref{GalLcountable}]\label{introGalLcountable}
The $\ell$-Galois exceptional locus $\ell-GalL$ is a countable union of closed algebraic subvarieties of $S$ defined over $\bar{K}$, and stable under the action of $\mathrm{Gal}(\bar{K}/K)$.
\end{thm}

This statement is implied by the Tate conjecture, since it implies that $\ell-GalL$ is dominated by a countable union of relative Hilbert schemes using an argument similar to the one in (\cite{Voisin}, 1.1).

According to Conjecture \ref{conjdel} and Conjecture \ref{conjtate}, all three loci $HL$, $AHL$ and $\ell-GalL$ should coincide.
We will prove that this equality of loci holds true if one restricts to the "positive dimensional" part of these loci and the derived group $G_S^{der}$ is simple.

We say that a closed irreducible subvariety $Z \subset S$ is \emph{of positive period dimension} if $H_Z \not=1$. This is equivalent to saying that the image of $Z$ under the complex period map is not a point.
Denote by $HL_{pos}$ (resp. $AHL_{pos}$, $\ell-GalL_{pos}$) the union of all strict special (resp. strict absolutely special, strict $\ell$-Galois special) subvarieties $Z$ of $S$ which are of positive period dimension.

\begin{thm}[see Theorem \ref{loci}]\label{lociintro}
Suppose that $G_S^{der}$ is simple.
Then $$HL_{pos} = AHL_{pos} = \ell-GalL_{pos}.$$
\end{thm}

The Theorem can be used to transfer recent results on the Hodge locus of positive period dimension (cf. \cite{KO}, \cite{BKU}) to the $\ell$-Galois exceptional locus of positive period dimension. For example, using \cite{BKU} we give a criterion for $\ell-GalL_{pos}$ to be a \emph{finite} union of $\ell$-Galois special subvarieties (Corollary \ref{ellGalLfiniteunion}).
Note that it also follows from the Theorem that $\ell-GalL_{pos}$ is independent of $\ell$, something which is not at all clear for the full $\ell$-Galois exceptional locus.

Theorem \ref{lociintro} demonstrates that we have good control over the special or $\ell$-Galois special subvarieties of positive period dimension (at least when $G_S^{der}$ is simple),
whereas we are not able to say anything about special points or $\ell$-Galois special points.

\subsection*{Applications to the Mumford-Tate conjecture}
Using the results \cite{BKU} of Baldi-Klingler-Ullmo on the structure of the Hodge locus of positive period dimension, we prove that the absolute Mumford-Tate conjecture \ref{absMT} holds for many projective hypersurfaces defined over transcendental extensions of $\mathbb{Q}$. We denote by $\mathcal{M}_{d,n}$ the moduli space of smooth hypersurfaces of degree $d$ in $\mathbb{P}^{n+1}$.

\begin{thm}[see Corollary \ref{corprojhyp}]\label{introcorprojhyp}
Assume that $n \ge 3$, $d \ge 5$, and $(n,d) \not= (4,5)$. There is a dense open $\mathbb{Q}$-subvariety $U \subset \mathcal{M}_{d,n}$ such that the absolute Mumford-Tate conjecture for $H^n_{\mathrm{prim}}$ holds for all $x \in U(\mathbb{C}) \setminus U(\bar{\mathbb{Q}})$.
\end{thm}

Finally, we describe the relation of the concept of $\ell$-Galois special subvarieties to the Mumford-Tate conjecture for abelian varieties.
For families of abelian varieties, Deligne has shown in (\cite{Deligne}, Theorem 2.11) that Conjecture \ref{conjdel} holds for all subvarieties $Z \subset S$. In this case, all special subvarieties are absolutely special and the Mumford-Tate conjecture is equivalent to the absolute Mumford-Tate conjecture.
In contrast to this, the Mumford-Tate conjecture for abelian varieties, while proven in a number of cases, is still open in general.
The existence of CM points in special subvarieties together with the fact that the Mumford-Tate conjecture is known for CM abelian varieties (\cite{CM}) shows that in this case, Conjecture \ref{everyGalspabssp} is in fact equivalent to Conjecture \ref{conjtate}, and thus to the Mumford-Tate conjecture.
We use monodromy arguments to show that Conjecture \ref{everyGalspabssp} can be reduced to the zero-dimensional case, i.e. to $\ell$-Galois special points.
\begin{samepage}
\begin{thm}[see Theorem \ref{mtequivalent}]\label{intromtequivalent}
The following are equivalent:

\begin{enumerate}[(i)]
\item
The Mumford-Tate conjecture \ref{mtconjecture} holds for all (principally polarized) abelian varieties of dimension $g$;
\item
Every $\ell$-Galois special subvariety of $\mathcal{A}_g$ is special;
\item
Every $\ell$-Galois special point $x \in \mathcal{A}_g(\bar{\mathbb{Q}})$ is a CM point.
\end{enumerate}
\end{thm}

As an application of Theorem \ref{mainthm} in the case of Shimura varieties we prove the following: \end{samepage}
\begin{cor}[see Corollary \ref{mGs}]
Let $Sh$ be a connected component of a Shimura variety of Hodge type such that $G_{Sh}^{der}$ is simple.
Let $Z \subset Sh$ be a strict $\ell$-Galois special subvariety which is positive dimensional and maximal for these properties.
Then
$G_{\ell, Z} = G_Z \otimes \mathbb{Q}_{\ell}$ and any $\ell$-Galois generic point of $Z$ satisfies the Mumford-Tate conjecture.
\end{cor}

The definition of $\ell$-Galois special subvarieties is somewhat indirect, as the maximal algebraic subvarieties with a given property.
To phrase more explicitly what we expect about the structure of these subvarieties, we want to give a conjectural description of the complete local ring $\widehat{\mathcal{O}_{Z,x}}$ of an $\ell$-Galois special subvariety $Z$ at a point $x \in Z(\bar{\mathbb{Q}})$, in the special case of a family of abelian varieties $f: \mathcal{X} \to S$ over $\bar{\mathbb{Q}}$.
For special subvarieties, such a local description follows from the theorem of Cattani-Deligne-Kaplan (cf. \cite{CDK}).

The infinitesimal variation of the Hodge filtration around the point $x$ can be described by a formal period map
$$\hat{\Phi}_{x}: \widehat{S_x} \to \widehat{\mathcal{F}_{y}} $$
defined over $\bar{\mathbb{Q}}$ to the completion of a flag variety $\mathcal{F}$ at a point $y$.

Using Fontaine's de Rham comparison isomorphism, we introduce a de Rham incarnation $G_{\ell-dR,Z}$ of the group $G_{\ell,Z}$, which gives a flag variety $\mathcal{F}_{\ell,Z} \subset \mathcal{F}$ defined over $\bar{\mathbb{Q}}_{\ell}$. By construction, the restriction of $\hat{\Phi}_x$ to the closed formal subscheme $\widehat{Z_x}$ factors through the completion $\widehat{(\mathcal{F}_{\ell,Z})_{y}} $ of $\mathcal{F}_{\ell,Z}$ at $y$.

\begin{conj}[see Conjecture \ref{ellCDK}]\label{introellCDK}
If $Z \subset S$ is an $\ell$-Galois special subvariety, then $\widehat{Z_x}$ is an irreducible component of the pullback $\hat{\Phi}_{x}^{-1}(\widehat{(\mathcal{F}_{\ell,Z})_{y}}) $ of $\widehat{(\mathcal{F}_{\ell,Z})_{y}}$ under the period map.
\end{conj}

\begin{rem}
Intuitively, the conjecture says that locally around $x$, the subvariety $Z$ is defined by the condition that the de Rham tensors fixed by $G_{\ell-dR,Z}$ stay in the zeroth step of the Hodge filtration.
In other words, we expect that the locus in $\widehat{S_x}$ cut out by the Tate tensors defined by $G_{\ell,Z}$, a priori just a closed formal subvariety, is in fact the germ of an \emph{algebraic} subvariety of $S$, namely the $\ell$-Galois special subvariety $Z$.
\end{rem}

Note that this conjecture is implied by the Mumford-Tate conjecture, using the mentioned local description of special subvarieties. We prove the converse:

\begin{thm}[see Theorem \ref{ellCDKMT}]\label{introellCDKMT}
Conjecture \ref{introellCDK} for $S= \mathcal{A}_g$ implies the Mumford-Tate conjecture for principally polarized abelian varieties of dimension $g$.
\end{thm}

The proof uses that Conjecture \ref{introellCDK} allows us to identify the dimension of an $\ell$-Galois special subvariety $Z \subset \mathcal{A}_g$ with the dimension of the flag variety $\mathcal{F}_{\ell,Z}$, and that one can show that the dimension of this flag variety is zero exactly if the group $G_{\ell,Z}$ is a torus.
The Mumford-Tate conjecture for abelian varieties then follows using Theorem \ref{intromtequivalent}.

\begin{notations}

We assume all varieties in this paper to be reduced.
For $S$ an irreducible variety over a field $K \subset \mathbb{C}$, by a subvariety $Z \subset S$ we mean a closed irreducible complex subvariety.
If $Z$ is defined over an extension $L $ of $K$, we denote the associated $L$-variety by $Z_L$.
If a distinction is necessary, we will denote by $Z_{\mathbb{C}} = Z_L \otimes_L \mathbb{C}$ the corresponding complex variety.
Throughout this paper, by a Hodge cycle we mean a cycle of type $(0,0)$.
\end{notations}
\begin{ac}
The author would like to thank Bruno Klingler for many helpful comments and discussions.
\end{ac}

\section{Families of varieties and their realizations}

Let $f: \mathcal{X} \to S$ be a smooth projective morphism between smooth irreducible quasi-projective complex algebraic varieties.
In this section we recall the tools from Hodge theory, absolute Hodge cycles and $\ell$-adic local systems that we will use for studying the cohomology of this family of varieties.

\subsection{Variation of Hodge structure}

After fixing a degree $k \ge 0$, the Hodge theoretic realization of the family $f$ is a (pure, polarizable) $\mathbb{Z}$-variation of Hodge structure $(\mathbb{V}_{\mathbb{Z}}, \mathcal{V}, \nabla, F^{\bullet}) $ on $S$, where $\mathbb{V}_{\mathbb{Z}} = R_{prim}^kf^{an}_{*} \underline{\mathbb{Z}}$ is the local system attached to the primitive $k$-th cohomology,
$(\mathcal{V} = R_{prim}^kf_{*} \Omega_{X/S}, \nabla)$ the associated algebraic vector bundle with flat connection and $F^{\bullet}$ is the Hodge filtration of the vector bundle $\mathcal{V}$.
We denote the associated $\mathbb{Q}$-variation of Hodge structure by $\mathbb{V}:= \mathbb{V}_{\mathbb{Z}} \otimes_{\mathbb{Z}} \mathbb{Q}$.

For every point $s \in S(\mathbb{C})$, the fiber $\mathbb{V}_s$ is a polarizable $\mathbb{Q}$-Hodge structure.
Given $m,n \ge 0$, we define the tensor Hodge structure
$\mathbb{V}_{s}^{\otimes (m,n)} := \mathbb{V}_s^{\otimes m} \otimes \mathbb{V}_s^{\vee \otimes n}$ and $\mathbb{V}_s^{\otimes} := \bigoplus_{m,n \ge 0} \mathbb{V}_s^{\otimes (m,n)}$.

For an irreducible closed algebraic subvariety $Z \subset S$ we define its generic Mumford-Tate group with respect to the variation $\mathbb{V}$ as the group fixing all generic Hodge tensors over $Z$.
More precisely, after choosing a point $z \in Z(\mathbb{C})$, we let $\mathcal{H}_Z \subset \mathbb{V}_z^{\otimes} $ denote the subspace of $v \in \mathbb{V}_z^{\otimes}$ which extend to a global section of $\mathbb{V}^{\otimes}$ over a finite étale cover $Z'$ of $Z$ such that the fiber of this global section is a Hodge class at every point $t \in Z'(\mathbb{C})$.

\begin{mydef}
We define the \emph{generic Mumford-Tate group} $G_Z$ of $Z$ to be
the subgroup of $GL(\mathbb{V}_z)$ fixing the tensors in $\mathcal{H}_Z$.
\end{mydef}

It is well-known that $G_Z$ is a connected reductive group over $\mathbb{Q}$.
We will often use the following formulation of the theorem of the fixed part (\cite{Schmid}, Corollary 7.23):

\begin{thm}\label{fixedpartnonsm}
Let $\xi$ be a global section of $\mathbb{V}^{\otimes (m,n)}$ over a finite étale cover $Z'$ of a closed irreducible subvariety $Z$ of $S$ such that $\xi_s$ is a Hodge cycle for some point $s \in Z'(\mathbb{C})$. Then $\xi_t$ is Hodge for every point $t \in Z'(\mathbb{C})$.
\end{thm}
Note that the theorem above is usually just stated for a smooth base $Z$, but one can reduce to that case using the period map (cf. \cite{Abssppaper}, Theorem 1.14).

\subsection{Absolute Hodge cycles}

In this section we recall Deligne's notion of absolute Hodge cycle (cf. \cite{Deligne}).

Let $X$ be a smooth projective algebraic variety over $\mathbb{C}$. For any automorphism $\sigma \in \mathrm{Aut}(\mathbb{C}/\mathbb{Q})$, we define the conjugated variety $X^{\sigma}:= X \times_{\mathbb{C}, \sigma} \operatorname{Spec} \mathbb{C}$.

The natural map $X^{\sigma} \to X$ induces isomorphisms
$$\sigma_{dR}^*: H_{dR}^{2n}(X /\mathbb{C})(n) \overset{\sim}{\to} H_{dR}^{2n}(X^{\sigma} /\mathbb{C})(n)$$
and
$$ \sigma_{\acute{e}t}^{*}: H^{2n}_{\acute{e}t}(X, \mathbb{A}_{f})(n) \overset{\sim}{\to} H^{2n}_{\acute{e}t}(X^{\sigma}, \mathbb{A}_{f})(n). $$

Using the natural de Rham and étale comparison isomorphisms, a Hodge class $v \in H^{2n}_B(X, \mathbb{Q})(n)$ gives rise to a de Rham component $v_{dR} \in H_{dR}^{2n}(X /\mathbb{C})(n)$ and an étale component $v_{\acute{e}t} \in H^{2n}_{\acute{e}t}(X, \mathbb{A}_{f})(n)$.

\begin{mydef}[\cite{Deligne}, §2]
A Hodge class $v \in H^{2n}_B(X, \mathbb{Q})(n)$ is called \emph{absolute Hodge} if for every $\sigma \in \mathrm{Aut}(\mathbb{C}/\mathbb{Q})$, the conjugated de Rham and étale components
$$ \sigma_{dR}^{*}(v_{dR}) \in H^{2n}_{dR}(X^{\sigma} / \mathbb{C})(n) $$ and
$$ \sigma_{\acute{e}t}^{*}(v_{\acute{e}t}) \in H^{2n}_{\acute{e}t}(X^{\sigma}, \mathbb{A}_{f})(n) $$
are the de Rham and étale components of a single Hodge class in $H^{2n}_B(X^{\sigma}, \mathbb{Q})(n)$.
\end{mydef}

Consider a closed irreducible subvariety $Z \subset S$ and a point $z \in Z(\mathbb{C})$.
We define $\mathcal{AH}_Z^{\otimes (m,n)} \subset \mathbb{V}_z^{\otimes (m,n)}$ to be the subset of all $v \in \mathbb{V}_z^{\otimes (m,n)}$ which extend to a global section of $\mathbb{V}^{\otimes (m,n)}$ over a finite étale cover $Z'$ of $Z$ such that the fiber of this global section is an absolute Hodge class at every point $t \in Z'(\mathbb{C})$.
Set $\mathcal{AH}_Z := \bigoplus_{m,n \ge 0} \mathcal{AH}_Z^{\otimes (m,n)}$.
\begin{mydef}We define the \emph{generic absolute Mumford-Tate group} $G_{Z}^{AH}$ to be the subgroup of $GL(\mathbb{V}_z)$ fixing all tensors in $\mathcal{AH}_Z$.
\end{mydef}

Note that the absolute Hodge classes of $ \mathbb{V}_z^{\otimes (m,n)}$ can be interpreted as absolute Hodge classes on some power of the fiber $\mathcal{X}_z$.

\begin{rem}\label{grouppoint}
In the case that $Z=\{z\}$ is a point, the group $G_z^{AH}$ is the motivic Galois group of the motive for absolute Hodge cycles attached to the $k$-th primitive cohomology of the fiber $\mathcal{X}_z$, cf. (\cite{DM}, Proposition 6.22).
Indeed, the motivic Galois group is the (possibly non-connected) reductive group characterized by the fact that it fixes the absolute Hodge tensors in all tensor constructions built out of $H^k_{prim}(\mathcal{X}_z)$ and its dual. Hence we recover exactly the above definition of $G_z^{AH}$.
\end{rem}

Clearly, every absolute Hodge class is a Hodge class. It follows that we have the inclusion $\mathcal{AH}_Z \subset \mathcal{H}_Z$ and therefore $G_Z \subset G_Z^{AH}$.

The analog of the theorem of fixed part for absolute Hodge classes is Deligne's Principle B (cf. \cite{Deligne}, Theorem 2.12):
\begin{thm}\label{principlebnonsm}
Let $\xi$ be a global section of $\mathbb{V}$ over a finite étale cover $Z'$ of a closed irreducible subvariety $Z$ of $S$ such that $\xi_s$ is absolute Hodge for some point $s \in Z'(\mathbb{C})$. Then $\xi_t$ is absolute Hodge for every point $t \in Z'(\mathbb{C})$.
\end{thm}
Again, this theorem is usually phrased with $Z$ smooth, but we can reduce to that case using Theorem \ref{fixedpartnonsm} as in (\cite{Abssppaper}, Theorem 2.5).

\subsection{$\ell$-adic realization}

Suppose that $S = S_K \otimes_K \mathbb{C}$ and $\mathcal{X} = \mathcal{X}_K \otimes_K \mathbb{C}$ are defined over a field $K\subset \mathbb{C}$ finitely generated over $\mathbb{Q}$ such that $f$ comes from a family $f_K: \mathcal{X}_K \to S_K$ defined over $K$.
Let $\mathbb{L}$ be the arithmetic étale $\ell$-adic local system on $S_K$ defined by $\mathbb{L}:= R_{prim}^kf_{K,*} \underline{\mathbb{Q}_{\ell}}$.
After fixing a point $s \in S(\bar{K})$, we denote the corresponding monodromy representation of the étale fundamental group by
$$\rho_{\ell}: \pi_1^{\acute{e}t}(S_K,s) \to GL(\mathbb{L}_s).$$
We use the notation $\overline{\mathbb{L}}$ for the $\ell$-adic étale local system on $S_{\bar{K}}$ defined by restricting this monodromy representation to $ \pi_1^{\acute{e}t}(S_{\bar{K}},s) $.
If $Z \subset S$ is an irreducible closed complex algebraic subvariety,
then $Z$ may be defined over a finitely generated field extension $L$ of $K$.
For a point $z \in Z(\bar{L})$, the restriction of the local system $\mathbb{L}$ on $S_K$ to $Z_L$ gives rise to a representation
\begin{equation}\label{rep} \rho_{\ell,Z}: \pi_1^{\acute{e}t}(Z_L,z) \to GL(\mathbb{L}_z). \end{equation}

\begin{lem}
The identity component of the Zariski closure of the image of (\ref{rep}) is independent of the choice of the field $L$.
\end{lem}
\begin{proof}
We recall the argument from (\cite{Moonen}, Remark 2.2.2).
Suppose $L'/L$ is a finitely generated field extension, and let $E := L' \cap \bar{L}$. Then $E$ is a finite extension of $L$.
Note that $$\pi_1^{\acute{e}t}(Z_{L'}) \twoheadrightarrow \pi_1^{\acute{e}t}(Z_E) \subset \pi_1^{\acute{e}t}(Z_L).$$
As a consequence, the Zariski closure of the image of $\pi_1^{\acute{e}t}(Z_{L'})$ is a finite index subgroup of the Zariski closure of the image of $\pi_1^{\acute{e}t}(Z_L)$.
\end{proof}

\begin{mydef}
The \emph{generic $\ell$-adic algebraic Galois group} $G_{\ell,Z}$ of $Z$ is the identity component of the Zariski closure in $GL(\mathbb{L}_z)$ of the image of the arithmetic monodromy representation
\begin{equation} \rho_{\ell,Z}: \pi_1^{\acute{e}t}(Z_L,z) \to GL(\mathbb{L}_{z}), \end{equation}
where $L$ is a field of definition of $Z$ which is finitely generated over $K$.
\end{mydef}

Here we use the same terminology as in (\cite{Moonen}, 4.2).
We warn the reader that it is not known in general whether $G_{\ell, Z}$ is a reductive group. This is essentially the statement of the semisimplicity conjecture, which is known only in specific cases, such as families of abelian varieties by the work of Faltings in \cite{Faltings}.

\begin{prop}\label{conjgalsp}
For every $\sigma \in \mathrm{Aut}(\mathbb{C}/K)$ there exists a natural isomorphism $G_{\ell,Z^{\sigma}} \cong G_{\ell,Z}$.
\end{prop}
\begin{proof}
If $Z$ is defined over a field $L$ finitely generated over $K$, the subvariety $Z^{\sigma}$ is defined over the finitely generated field extension $\sigma(L)$ of $K$, and $\sigma$ defines an isomorphism of $K$-schemes $Z_{\sigma(L)}^{\sigma} \overset{\sim}{\rightarrow} Z_L$.
The diagram
$$\xymatrix{\pi_1^{\acute{e}t}(Z_{\sigma(L)}^{\sigma}, z^{\sigma}) \ar[r] \ar[d]^{\cong} & \pi_1^{\acute{e}t}(S_K, z^{\sigma}) \ar[r] \ar[d]^{\cong}& GL(\mathbb{L}_{z^{\sigma}}) \ar[d]^{\cong} \\
\pi_1^{\acute{e}t}(Z_L, z) \ar[r] & \pi_1^{\acute{e}t}(S_K, z) \ar[r] & GL(\mathbb{L}_{z})
} $$
commutes because the local system $\mathbb{L}$ is defined over $K$.
Thus under the isomorphism $\mathbb{L}_{z^{\sigma}} \cong \mathbb{L}_{z}$ induced by $\sigma$ the images of the above representations coincide, hence so do their Zariski closures.
\end{proof}

\subsection{Monodromy}

A powerful tool relating the Hodge theoretic and $\ell$-adic realizations is the notion of monodromy.
Let $Z \subset S$ be a closed irreducible subvariety and $z \in Z(\mathbb{C})$.

\begin{mydef}\label{algmon1}
The \emph{algebraic monodromy group} $H_Z$ of $Z$ is defined to be the identity component of the Zariski closure of the image of the monodromy representation
$ \rho_Z: \pi_1(Z_{\mathbb{C}}^{an},z) \to GL(\mathbb{V}_z)$ corresponding to the restriction of $\mathbb{V}$ to $Z$.
\end{mydef}

\begin{rem}\label{notsmooth}
The algebraic monodromy group is often defined by restricting to the smooth locus $Z^{sm}$ of $Z$.
Using Hodge theory, one can show that $H_{Z^{sm}}=H_Z$, i.e. both definitions agree (cf. \cite{Abssppaper}, Lemma 1.9).
\end{rem}

Let $L$ be a field of definition for $Z$ which is finitely generated over $K$.

\begin{mydef}\label{algmon2}
The \emph{geometric $\ell$-adic algebraic monodromy group} $H_{\ell,Z}$ of $Z$ is defined to be the identity component of the
Zariski closure of the image of the monodromy representation
$\overline{\rho}_{\ell, Z}: \pi_1^{\acute{e}t}(Z_{\bar{L}},z) \to GL(\overline{\mathbb{L}}_z)$
attached to the restriction of the $\ell$-adic local system $\overline{\mathbb{L}}$ to $Z_{\bar{L}}$.
\end{mydef}

Note that the common geometric origin of the local systems gives an identification $\mathbb{V} \otimes_{\mathbb{Q}} \mathbb{Q}_{\ell} \cong \nu^{*}\overline{\mathbb{L}}$ under the morphism of topoi $$(\nu^{*}, \nu_{*}): S_{\mathbb{C}}^{an} \to S_{\mathbb{C}, \acute{e}t} \cong S_{\bar{K}, \acute{e}t}.$$
As a consequence, the following proposition relates the algebraic monodromy group and the geometric $\ell$-adic algebraic monodromy group. This also shows that Definition \ref{algmon2} is again independent of the choice of the field $L$.

\begin{prop}[\cite{Moonen}, Lemma 4.3.4]\label{monodromyladic}
The $\ell$-adic comparison isomorphism $\mathbb{V}_z \otimes_{\mathbb{Q}} \mathbb{Q}_{\ell} \cong \overline{\mathbb{L}}_z$
induces an isomorphism
$$ H_{Z} \otimes_{\mathbb{Q}} \mathbb{Q}_{\ell} \cong H_{\ell, Z}.$$
\end{prop}
\begin{proof}
By definition, $H_{\ell,Z}$ is the identity component of the Zariski closure of the image of the monodromy representation $\overline{\rho}_{\ell,Z}: \pi_1^{\acute{e}t}(Z_{\bar{L}},z) \to GL(\overline{\mathbb{L}}_z)$.
Since the natural map $\pi_1^{\acute{e}t}(Z_{\mathbb{C}},z) \twoheadrightarrow \pi_1^{\acute{e}t}(Z_{\bar{L}},z)$ is surjective, it is also the identity component of the Zariski closure of the monodromy representation $\rho_{\ell, Z_{\mathbb{C}}}: \pi_1^{\acute{e}t}(Z_{\mathbb{C}},z) \to GL(\overline{\mathbb{L}}_z)$.
We now recall the argument given in (\cite{Moonen}, Lemma 4.3.4).
We have a commutative diagram
$$\xymatrix{\pi_1(Z_{\mathbb{C}}^{an},z) \ar[r]^{\rho_Z} \ar[d] & GL(\mathbb{V}_z \otimes \mathbb{Q}_{\ell}) \ar[d]^{\cong} \\
\pi_1^{\acute{e}t}(Z_{\mathbb{C}},z) \ar[r]^{\rho_{\ell,Z_{\mathbb{C}}}} & GL(\mathbb{L}_z) \,\, .
}
$$
It suffices to show that $\rho_{\ell, Z_{\mathbb{C}}}(\pi_1^{\acute{e}t}(Z_{\mathbb{C}},z)) \subset H_Z\otimes \mathbb{Q}_{\ell}$.
The étale fundamental group $\pi_1^{\acute{e}t}(Z_{\mathbb{C}},z)$ is the profinite completion of $\pi_1(Z_{\mathbb{C}}^{an},z)$ (\cite{Grothendieck} Exposé V, Corollary 5.2).
We see that $\rho_Z(\pi_1(Z_{\mathbb{C}}^{an},z))$ is dense in $\rho_{\ell, Z_{\mathbb{C}}}(\pi_1^{\acute{e}t}(Z_{\mathbb{C}},z))$ for the $\ell$-adic topology, and therefore also for the Zariski topology.
\end{proof}

\begin{prop} \label{normal}
\textnormal{ }
\begin{enumerate}[(i)]
\item
The algebraic monodromy group $H_Z$ is a normal subgroup of the derived group of the generic Mumford-Tate group $G_Z$, and of the generic absolute Mumford-Tate group $G_Z^{AH}$.
\item \label{b}
The geometric $\ell$-adic algebraic monodromy group $H_{\ell, Z}$ is a normal subgroup of the generic $\ell$-adic algebraic Galois group $G_{\ell,Z}$.

\end{enumerate}
\end{prop}
\begin{proof}
Arguing as in (\cite{Abssppaper}, Proposition 2.9), we may restrict to the smooth locus of $Z$, and therefore assume that $Z$ is smooth.
The first part is (\cite{Andre}, Theorem 1), and the corresponding analog for absolute Hodge cycles follows from Deligne's Principle B (\cite{Deligne}, Theorem 2.12).
We now come to the $\ell$-adic part (\ref{b}).
Suppose that $Z$ is defined over a field $L$ finitely generated over $K$.
Choose a geometric point $\bar{z} \in Z(\bar{L})$.
The short exact sequence $$ 1 \to \pi_1^{\acute{e}t}(Z_{\bar{L}}, \bar{z}) \to \pi_1^{\acute{e}t}(Z_L, \bar{z}) \to \mathrm{Gal}(\bar{L}/L) \to 1$$
shows that $H_{\ell, Z}$ is a normal subgroup of $G_{\ell, Z}$.
\end{proof}

The following Lemma is a geometric formulation of two fundamental principles in Hodge theory: the theorem of the fixed part (cf. Theorem \ref{fixedpartnonsm}) and Deligne's Principle B (cf. Theorem \ref{principlebnonsm}).

\begin{lem}\label{lem}
Let $Z$ and $Y$ be closed irreducible subvarieties of $S$.
\begin{enumerate}[(i)]
\item
Suppose $Z \subset Y$ and $H_Y \subset G_Z$, then $G_Z = G_Y$.
\item
Suppose $Z \subset Y$ and $H_Y \subset G^{AH}_Z$, then $G^{AH}_Z = G^{AH}_Y$.
\item\label{c}
Suppose $Z \subset Y$ and $H_{\ell, Y} \subset G_{\ell, Z}$, then $G_{\ell, Z} = G_{\ell, Y}$.
\end{enumerate}
\end{lem}
\begin{proof}
\begin{enumerate}[(i)]
\item
We show that $G_Y \subset G_Z$, the other inclusion being clear.
Let $z \in Z(\mathbb{C})$ and suppose $v \in \mathbb{V}_z^{\otimes}$ is fixed by $G_{Z}$. Then up to replacing $Y$ by a finite étale cover, the condition $H_{Y} \subset G_{Z}$ will ensure that $v$ extends to a global section on $Y$. But $v$ is Hodge at the point $z \in Z(\mathbb{C})$, hence is a generic Hodge tensor over $Y$ by the Theorem of the fixed part (Theorem \ref{fixedpartnonsm}). We conclude that $G_{Y}$ fixes $v$.
\item
Similarly, if $v \in \mathbb{V}_z^{\otimes}$ is fixed by $G_Z^{AH}$, up to replacing $Y$ by a finite étale cover the condition $H_{Y} \subset G_{Z}^{AH}$ shows that $v$ extends to a global section on $Y$. We know that it is absolute Hodge at the point $z \in Z(\mathbb{C})$.
Now we use Deligne's Principle B (Theorem \ref{principlebnonsm}) to conclude that a global section which is absolute Hodge at one point is absolute Hodge everywhere, hence fixed by $G_Y^{AH}$.
\item
We may choose a finitely generated field extension $L$ of $K$ such that $Z$ and $Y$ are both defined over $L$ and there exists an $L$-rational point $z \in Z(L)$. Denote by $\bar{z}$ a geometric point lying over $z$.
The point $z$ induces a splitting of the homotopy short exact sequence
$$ 1 \to \pi_1^{\acute{e}t}(Z_{\bar{L}}, \bar{z}) \to \pi_1^{\acute{e}t}(Z_L, \bar{z}) \to Gal(\bar{L}/L) \to 1,$$
giving a description $$\pi_1^{\acute{e}t}(Z_L, \bar{z}) = \pi_1^{\acute{e}t} (Z_{\bar{L}}, \bar{z}) \rtimes \mathrm{Gal}(\bar{L}/L)$$ as a semi-direct product.
This shows that $G_{\ell,Z}$ is generated by $H_{\ell,Z}$ and $G_{\ell,z}$, the identity component of the Zariski closure of the image of the Galois representation attached to the point $z$.
Similarly, $G_{\ell, Y}$ is generated by $H_{\ell, Y}$ and $G_{\ell,z}$. From this we see that $H_{\ell,Y} \subset G_{\ell,Z}$ implies $G_{\ell, Z} = G_{\ell, Y}$.
\end{enumerate}

\end{proof}

\begin{prop}
We have the inclusion $G_{\ell,Z} \subset G_Z^{AH}\otimes \mathbb{Q}_{\ell}$.
\end{prop}
\begin{proof}
This is an adaptation of the argument given in (\cite{Deligne}, Proposition 2.9(b)).
Let $L$ be a field of definition of $Z$ which is finitely generated over $K$. Up to replacing $L$ by a finite extension, we may assume the existence of a point $z \in Z(L)$. Then $\mathrm{Gal}(\bar{L}/L)$ acts on the fiber $\mathbb{L}_{\bar{z}}$ over a geometric point $\bar{z} \in Z(\bar{L})$ lying over $z$.
We claim that $\mathcal{AH}^{\otimes (m,n)}_Z \subset \mathbb{L}^{\otimes (m,n)}_{\bar{z}}$ is preserved by $\mathrm{Gal}(\bar{L}/L)$. Indeed, any element in $\mathrm{Gal}(\bar{L}/L)$ can be extended to an automorphism in $\mathrm{Aut}(\mathbb{C}/\mathbb{Q})$, and these preserve $\mathcal{AH}^{\otimes (m,n)}_Z$ by definition.
Now the profinite group $\mathrm{Gal}(\bar{L}/L)$ acts continuously on the finite dimensional $\mathbb{Q}$-vector space $\mathcal{AH}^{\otimes (m,n)}_Z$ and thus acts through a finite quotient.
After again replacing $L$ by a finite extension we may assume that $\mathrm{Gal}(\bar{L}/L)$ acts trivially.
This shows that the identity component $G_{\ell,z}$ of the Zariski closure of the Galois representation attached to the point $z$ is contained in $G_Z^{AH}\otimes \mathbb{Q}_{\ell}$, as the latter is defined as the group fixing the elements in $\mathcal{AH}_Z$.
Since we know from Proposition \ref{monodromyladic} that $H_{\ell,Z} \subset G_Z^{AH} \otimes \mathbb{Q}_{\ell}$, arguing as in the proof of Lemma \ref{lem}(\ref{c}) we can show that $G_Z^{AH} \otimes \mathbb{Q}_{\ell}$ contains $G_{\ell,Z}$.
\end{proof}

\subsection{Relation to the usual Mumford-Tate conjecture}\label{usualMT}

We briefly describe the relation of Conjecture \ref{conjtate} to the usual absolute Mumford-Tate conjecture \ref{absMT}.
Let $Z\subset S$ be a closed irreducible subvariety.

\begin{mydef}
A point $z \in Z(\mathbb{C})$ is called \emph{$\ell$-Galois generic} if $G_{\ell,z} = G_{\ell, Z}$.
\end{mydef}

\begin{lem}\label{existencegalgen}
Let $Z$ be defined over a field $L$ finitely generated over $K$.
If $Z^{sm}$ denotes the smooth locus of $Z$, we have $G_{\ell, Z^{sm}} = G_{\ell, Z}$.
As a consequence, if $Z$ is positive dimensional, there is a number $d \ge 0$ such that there are infinitely many $\ell$-Galois generic points of $Z$ defined over extensions of degree $\le d$ over $L$.
\end{lem}
\begin{proof}
Using methods from Hodge theory one can show that $H_Z = H_{Z^{sm}}$ (cf. \cite{Abssppaper}, Lemma 1.9), hence also $H_{\ell,Z} = H_{\ell,Z^{sm}}$ by the arguments in the proof of Proposition \ref{monodromyladic}.
Arguing as in the proof of Lemma \ref{lem}(\ref{c}) it follows that $G_{\ell, Z^{sm}} = G_{\ell,Z}$.
For the smooth part
it is known (cf. \cite{Serre}, Section 10.6) that there are infinitely many points defined over extensions of bounded degree such that $G_{\ell,z} =G_{\ell,Z^{sm}} =G_{\ell, Z}$.
\end{proof}

We can now show that Conjectures \ref{conjdel} and \ref{conjtate} are in fact equivalent to the usual conjectures \ref{HAH} and \ref{absMT} over a point.
Namely, for a given subvariety $Z \subset S$, we may choose an $\ell$-Galois generic point $z \in Z(\mathbb{C})$.
Then $H_{\ell,Z} \subset G_{\ell,Z}=G_{\ell,z} \subset G_z^{AH} \otimes \mathbb{Q}_{\ell}$ and by Lemma \ref{lem} we have $G_z^{AH}=G_Z^{AH}$.
Assuming that Conjecture \ref{conjtate} holds true for the point $z$, we obtain the desired equality $$G_{\ell, Z} = G_{\ell, z} =G_z^{AH} \otimes \mathbb{Q}_{\ell} = G_Z^{AH}\otimes \mathbb{Q}_{\ell}.$$
The argument for Conjecture \ref{conjdel} is similar.

\section{$\ell$-Galois special subvarieties}

In this section we recall the notion of special subvariety as a subvariety which is maximal with a given generic Mumford-Tate group.
Motivated by Conjectures \ref{conjdel} and \ref{conjtate} we introduce a notion of absolutely special and $\ell$-Galois special subvarieties, replacing the generic Mumford-Tate group by the generic absolute Mumford-Tate group and the generic $\ell$-adic algebraic Galois group.

\subsection{Definition}

We define special, absolutely special and $\ell$-Galois special subvarieties and prove their first properties.

\begin{mydef}[\cite{KO}, Definition 1.2; \cite{Abssppaper}, Definition 0.4]\label{definitionspecial}
\textnormal{ }
\begin{enumerate}[(i)]
\item
A closed irreducible subvariety $Z \subset S$ is called \emph{special}
if it is maximal among the closed irreducible subvarieties of $S$ having the same generic Mumford-Tate group as $Z$.
\item
A closed irreducible subvariety $Z \subset S$ is called \emph{absolutely special}
if it is maximal among the closed irreducible subvarieties of $S$ having the same generic absolute Mumford-Tate group as $Z$.
\end{enumerate}
\end{mydef}

A closely related variant of the definition of absolutely special subvarieties, called dR-absolutely special subvarieties, was given in (\cite{Abssppaper}, Definition 0.4).

We define the following $\ell$-adic incarnation of special subvarieties.

\begin{mydef} \label{defGalsp}

A closed irreducible subvariety $Z \subset S$ is called \emph{$\ell$-Galois special} if it is maximal among the closed irreducible subvarieties $Y$ of $S$ whose generic $\ell$-adic algebraic Galois group $G_{\ell,Y}$ equals $G_{\ell,Z}$.
\end{mydef}

In other words, $Z \subset S$ is $\ell$-Galois special if there does not exist a closed irreducible subvariety $Y$ of $S$ containing $Z$ strictly with $G_{\ell,Y} = G_{\ell, Z}$.

We prove that an absolutely special subvariety is both special and $\ell$-Galois special for all primes $\ell$.

\begin{prop} \label{absspecial}
Let $Z$ be an absolutely special subvariety.
\begin{enumerate}[(i)]
\item \label{both}
$Z$ is both special and $\ell$-Galois special.
\item \label{conj}
$Z^{\sigma}$ is absolutely special for all $\sigma \in \mathrm{Aut}(\mathbb{C} / K)$.
\item \label{three}
$Z$ is defined over $\bar{K}$ and all its Galois conjugates are special.
\end{enumerate}
\end{prop}
\begin{proof}
For (\ref{both}), let us prove that $Z$ is $\ell$-Galois special. The proof for special is similar (cf. \cite{Abssppaper}, Proposition 3.2).
Suppose there exists a closed irreducible subvariety $Y \supset Z$ such that $G_{\ell, Z} = G_{\ell,Y}$.
It follows that $$H_{\ell,Y} \subset G_{\ell,Y} = G_{\ell,Z} \subset G^{AH}_{Z}\otimes \mathbb{Q}_{\ell}.$$ Using Proposition \ref{monodromyladic} this in turn shows that $H_Y \subset G_Z^{AH}$. Lemma \ref{lem} then says that $G_Y^{AH} = G_Z^{AH}$. As $Z$ is absolutely special this implies $Y=Z$. Hence $Z$ is $\ell$-Galois special.
Since the action of $\mathrm{Aut}(\mathbb{C}/K)$ on $\ell$-adic étale cohomology maps absolute Hodge classes to absolute Hodge classes, there is a natural isomorphism $G^{AH}_Z \otimes \mathbb{Q}_{\ell} \cong G^{AH}_{Z^{\sigma}} \otimes \mathbb{Q}_{\ell}$. This proves (\ref{conj}).
Now (\ref{three}) follows from (\ref{both}) and the fact that there are only countably many special subvarieties.
\end{proof}

It follows from Conjectures \ref{conjdel} and \ref{conjtate} that the notions of special, $\ell$-Galois special and absolutely special subvarieties should all be equivalent. The fact that we can only prove that absolutely special implies special and $\ell$-Galois special corresponds to the caveat that we only know the two inclusions $G_Z \subset G_Z^{AH}$ and $G_{\ell,Z} \subset G_Z^{AH} \otimes_{\mathbb{Q}} \mathbb{Q}_{\ell}$.

\subsection{Fields of definition of $\ell$-Galois special subvarieties}

We prove that $\ell$-Galois special subvarieties are defined over $\bar{K}$, and their Galois conjugates are again $\ell$-Galois special.

\begin{prop}\label{conjGalsp}
If $Z\subset S$ is an $\ell$-Galois special subvariety then for every $\sigma \in \mathrm{Aut}(\mathbb{C}/K)$, the conjugate $Z^{\sigma} \subset S$ is an $\ell$-Galois special subvariety.
\end{prop}
\begin{proof}
Consider a closed irreducible subvariety $Y \supset Z^{\sigma}$ that satisfies $G_{\ell,Y}=G_{\ell,Z^{\sigma}}$. Proposition \ref{conjgalsp} shows that for the conjugates $Y^{\sigma^{-1}} \supset Z$ we have the equality
$G_{\ell,Y^{\sigma^{-1}}} = G_{\ell,Z} $. Since $Z$ is assumed to be $\ell$-Galois special, it follows that $ Y^{\sigma^{-1}} = Z$. This proves that $Y = Z^{\sigma}$ and $Z^{\sigma}$ is $\ell$-Galois special.
\end{proof}

However, Conjecture \ref{everyGalspabssp} predicts more, namely that all $\ell$-Galois special subvarieties are defined over finite extensions of $K$.

\begin{prop}
Let $Z \subset S$ be a closed irreducible subvariety and let $Y$ denote the smallest closed $\bar{K}$-subvariety of $S$ containing $Z$. Then $G_{\ell,Z}= G_{\ell,Y}$.
\end{prop}
\begin{proof}
By Lemma \ref{existencegalgen} we may replace $Y$ by its smooth locus $Y^{sm}$ and $Z$ by the smooth locus of $Z \cap Y^{sm}$. Here we note that the last intersection is non-empty by the assumption that $Y$ is the smallest closed $\bar{K}$-subvariety containing $Z$.
We may therefore assume that $Z$ and $Y$ are both smooth.
Choose a finite extension $E$ of $K$ such that $Y$ is defined over $E$ and irreducible as an $E$-scheme, and choose a field of definition $L$ of $Z$ which is finitely generated over $E$.
By the fact that $Z$ is $\bar{K}$-Zariski dense in $Y$, the generic point $\eta_Z$ of $Z_L$ maps to the generic point $\eta_{Y}$ of $Y_E$. Choose geometric points $\bar{\eta}_Y$ and $\bar{\eta}_Z$ lying over $\eta_Y$ respectively $\eta_Z$.
Since $Z_L$ and $Y_E$ are smooth integral schemes, by (\cite{Grothendieck} Exposé V, Proposition 8.2) there is a surjection of étale fundamental groups $$\mathrm{Gal}(\overline{\kappa(\eta_Z)} / \kappa(\eta_Z)) = \pi_1^{\acute{e}t}(\eta_Z , \bar{\eta}_Z) \twoheadrightarrow \pi_1^{\acute{e}t}(Z_L, \bar{\eta}_Z)$$ and
$$\mathrm{Gal}(\overline{\kappa(\eta_Y)} / \kappa(\eta_Y)) = \pi_1^{\acute{e}t}(\eta_Y , \bar{\eta}_Y) \twoheadrightarrow \pi_1^{\acute{e}t}(Y_E, \bar{\eta}_Y).$$
Hence $G_{\ell,Z}$ can be identified with the identity component of the Zariski closure of the image of $\mathrm{Gal}(\overline{\kappa(\eta_Z)} / \kappa(\eta_Z))$, and similarly for $G_{\ell,Y}$.
Since $\kappa(\eta_Z)$ and $\kappa(\eta_Y)$ are finitely generated fields, the argument in (\cite{Moonen}, Remark 2.2.2(i)) shows that
$$\mathrm{Gal}(\overline{\kappa(\eta_Z)} / \kappa(\eta_Z)) \to \mathrm{Gal}(\overline{\kappa(\eta_Y)} / \kappa(\eta_Y)) $$ is a surjection onto a finite index subgroup.
This proves that $G_{\ell,Z} = G_{\ell,Y}$.
\end{proof}

\begin{cor}\label{fodgalsp}
Let $Z\subset S$ be an $\ell$-Galois special subvariety. Then $Z$ is defined over $\bar{K}$ and all $\mathrm{Gal}(\bar{K} / K)$-conjugates of $Z$ are again $\ell$-Galois special.
\end{cor}

\subsection{Weakly special subvarieties}

We recall the definition of weakly special subvarieties, defined using the algebraic monodromy group $H_Z$.

\begin{mydef}[\cite{KO}, Corollary 4.14]
A closed irreducible subvariety $Z \subset S$ is called \emph{weakly special}
if it is maximal among the closed irreducible subvarieties of $S$ having the same algebraic monodromy group as $Z$.
\end{mydef}

The relevance of this definition for us is that the class of weakly special subvarieties contains all the previously defined subvarieties.

\begin{lem}\label{sabGsws}
Any special, absolutely special or $\ell$-Galois special subvariety is weakly special.
\end{lem}
\begin{proof}
Let $Z$ be a special (resp. absolutely special, $\ell$-Galois special) subvariety. Let $Y$ be a closed irreducible subvariety of $S$ such that $Z \subset Y$ and $H_{Y}=H_{Z}$.
By Lemma \ref{lem} this implies $G_Z = G_Y$ (resp. $G^{AH}_Z =G^{AH}_Y$, $G_{\ell, Z} = G_{\ell, Y}$). Since $Z$ is special (resp. absolutely special, $\ell$-Galois special) we conclude that $Y=Z$. This proves that $Z$ is weakly special.
\end{proof}

Note that this allows us, in principle at least, to study aspects of the geometry of $\ell$-Galois special subvarieties using methods from Hodge theory. For example, weakly special subvarieties have a nice description using the period map.
Namely, they can be understood as irreducible components of preimages under the period map of certain group-theoretically defined subspaces of the period domain.

Let
$$\Phi: S^{an} \to X:= \Gamma\backslash \mathcal{D}$$
be the period map attached to the variation of Hodge structure $\mathbb{V}$.
Here $\mathcal{D}$ is the $G_S(\mathbb{R})$-orbit of a Hodge cocharacter $h: \mathbb{S} \to G_{S,\mathbb{R}}$ and $\Gamma \subset G_S(\mathbb{Q})$ is a subgroup which contains the image of the monodromy representation. The pair $(G_S, \mathcal{D})$ is called a Hodge datum in (\cite{KO}, Section 4) and the quotient $X= \Gamma\backslash \mathcal{D}$ a Hodge variety.
A morphism of Hodge data from $ (G', \mathcal{D}')$ to $(G, \mathcal{D})$ is a group homomorphism $\iota: G' \to G$ such that $\iota(\mathcal{D}') \subset \mathcal{D}$.
Provided that $ \iota(\Gamma') \subset \Gamma$, it induces a morphism of Hodge varieties $X' = \Gamma' \backslash \mathcal{D}' \to X$.

\begin{mydef}[\cite{KO}, Definition 4.1]
Let $ X \overset{\iota}{\leftarrow} X_1 \overset{\pi}{\rightarrow} X_2$ be a diagram of morphisms of Hodge varieties and $x_2 \in X_2$. An irreducible component of $\iota(\pi^{-1}(\{x_2\}))$ is
called a \emph{weakly special subvariety} of the Hodge variety $X$.
\end{mydef}

\begin{thm}[\cite{KO}, Corollary 4.14]\label{wsperiodmap}
The weakly special subvarieties of $S$ are precisely the irreducible components of the preimages of weakly special subvarieties in $X$ under the period map.
\end{thm}

\begin{ex}\label{shimura}
We illustrate the situation with the description of weakly special subvarieties of Shimura varieties, cf. (\cite{KUY}, Theorem 3.1).
Let $S:=Sh_G$ be a connected Shimura variety attached to a (connected) Shimura datum $(G, \mathcal{D})$. The special subvarieties of $Sh_G$ are given by the irreducible components of images of sub-Shimura varieties $Sh_{G'}$ associated with sub-Shimura data $(G', \mathcal{D}')$.
The weakly special subvarieties can be described as follows:
Given such a sub-Shimura datum $(G', \mathcal{D}')$, let $H$ be a normal subgroup of $G'$. The quotient Shimura datum $(G'/H, \overline{\mathcal{D}'})$ gives rise to a morphism of Shimura varieties $\pi: Sh_{G'} \to Sh_{G'/H}$. For any point $x \in Sh_{G'/H}(\mathbb{C})$,
the image of $ \pi^{-1}(\{x\})$ in $Sh_G$ is a weakly special subvariety, and all weakly special subvarieties are of this form.
In the Shimura variety case, the weakly special subvarieties were described by Moonen in (\cite{MoonenLinearity}, §3) and identified with the totally geodesic subvarieties (\cite{MoonenLinearity}, Theorem 4.3).
\end{ex}

\subsection{Weakly non-factor subvarieties}

In this section we prove the equivalence of the notions of special, absolutely special and $\ell$-Galois special subvarieties for the so-called \emph{weakly non-factor} subvarieties introduced in \cite{KOU}.
The main point is that in this case all three notions are equivalent to being weakly special, a notion which relies purely on monodromy.

We recall the notion of weakly non-factor subvarieties as defined in \cite{KOU}.
\begin{mydef}[\cite{KOU}, Definition 1.10]
A closed irreducible subvariety $Z \subset S$ is called \emph{weakly non-factor} if it is not contained in a closed irreducible subvariety $Y \subset S$ such that $H_{Z}$ is a strict normal subgroup of $H_{Y}$.
\end{mydef}

\begin{prop}\label{wswnf}
Assume $Z \subset S$ is weakly special and weakly non-factor. Then $Z$ is absolutely special.
\end{prop}

\begin{proof}
Assume $Y \subset S$ is a closed irreducible subvariety containing $Z$ such that $G^{AH}_{Y} = G^{AH}_{Z}$. As the monodromy groups $H_Z$ and $H_{Y}$ are both normal subgroups of the common generic absolute Mumford-Tate group $G^{AH}_Y =G^{AH}_Z$, the group $H_{Z}$ is a normal subgroup of $H_{Y}$. But $Z$ is weakly non-factor by assumption, so $H_Z=H_Y$. Since $Z$ is weakly special it follows that we have the equality $Z=Y$.
\end{proof}

The same proof was used in (\cite{Abssppaper}, Theorem 3.9) to conclude that a weakly special and weakly non-factor subvariety is dR-absolutely special.

\begin{thm}\label{equiv}
For a weakly non-factor subvariety $Z$, the following are equivalent:

\begin{enumerate}[(i)]
\item \label{ws}
$Z$ is weakly special;
\item \label{s}
$Z$ is special;
\item \label{Gs}
$Z$ is $\ell$-Galois special;
\item \label{abs}
$Z$ is absolutely special.
\end{enumerate}

\end{thm}
\begin{proof}
We already know that $(\ref{abs}) \implies (\ref{ws}), (\ref{s}), (\ref{Gs})$ and $(\ref{s}), (\ref{Gs}) \implies (\ref{ws})$. Therefore it is enough to prove $(\ref{ws}) \implies (\ref{abs})$. This is precisely Proposition \ref{wswnf}.
\end{proof}
We give some examples of weakly non-factor subvarieties that were already considered in (\cite{KOU}, Corollary 1.13).

\begin{mydef}
A closed irreducible subvariety $Z \subset S$ is called of \emph{positive period dimension} if $H_Z \not = 1$.
\end{mydef}
The subvariety $Z$ is of positive period dimension if and only if its image under the complex period map is not a point.

\begin{cor}\label{application1}
Suppose that the derived group $G_{S}^{der}$ is simple.
Let $Z \subset S$ be a strict special subvariety which is of positive period dimension, and maximal for these properties.
Then $Z$ is absolutely special.
\end{cor}
Indeed, such a subvariety is weakly non-factor by (\cite{KOU}, proof of Corollary 1.13).

\begin{cor}\label{Corrolary}
Suppose that $G_{S}^{der}$ is simple.
Let $Z\subset S$ be a strict $\ell$-Galois special subvariety which is of positive period dimension, and maximal for these properties.
Then $Z$ is absolutely special.
\end{cor}
\begin{proof}
As $Z$ is $\ell$-Galois special, in view of Theorem \ref{equiv} it suffices to show that $Z$ is weakly non-factor. Let $Y \supset Z$ be a closed irreducible subvariety such that $H_{ Z}$ is a strict normal subgroup of $H_{Y}$. By the maximality of $Z$, we have the equality $G_{\ell,Y} = G_{\ell, S}$. Thus $H_{\ell,Y}$ is a normal subgroup of $G_{\ell, S}$, and therefore also of $H_{\ell,S}$.
By Proposition \ref{monodromyladic} we know that $H_Y$ is a normal subgroup of $H_S$, which is simple by Proposition \ref{normal}.
As $H_{Z}$ is non-trivial, it follows that $H_{Z}=H_{Y}= H_S$, which contradicts the assumption that $H_{Z}$ is a strict normal subgroup of $H_{Y}$. Therefore $Z$ is weakly non-factor, as desired.
\end{proof}

\section{Applications to the Mumford-Tate conjecture}

We give a few instances of the interaction of the concept of $\ell$-Galois special subvarieties with the Mumford-Tate conjecture.
First we show that the $\ell$-Galois exceptional locus is a countable union of algebraic subvarieties defined over $\bar{K}$, and its part of positive period dimension coincides with the Hodge locus of positive period dimension provided that $G_S^{der}$ is simple.
Together with the results of \cite{BKU}, we use this to prove that the absolute Mumford-Tate conjecture holds for many smooth projective hypersurfaces defined over finitely generated transcendental extensions of $\mathbb{Q}$. Secondly, we outline how the question of determining the $\ell$-Galois special points in the moduli space of (principally polarized) abelian varieties is equivalent to the Mumford-Tate conjecture for abelian varieties.

\subsection{Exceptional loci of positive period dimension}

We begin by defining the following exceptional loci:

\begin{mydef} \textnormal{ }
\begin{enumerate}[(i)]
\item
The \emph{Hodge locus} $HL$ of $S$ is the set of $s \in S(\mathbb{C})$ such that $G_s \subsetneq G_S$.
\item
The \emph{absolute Hodge locus} $AHL$ is the set of all $s \in S(\mathbb{C})$ such that $G_s^{AH} \subsetneq G_S^{AH}$.
\item
The \emph{$\ell$-Galois exceptional locus} $\ell-GalL$ is the set of all $s \in S(\mathbb{C})$ such that $G_{\ell, s} \subsetneq G_{\ell,S}$.
\end{enumerate}
\end{mydef}

\begin{prop}\label{lociunion}
The Hodge locus $HL$ (resp. $AHL$, $\ell-GalL$) is the union of all strict special (resp. strict absolutely special, strict $\ell$-Galois special) subvarieties of $S$.
\end{prop}
\begin{proof}
We prove the claim for $\ell-GalL$.
Let $s \in \ell-GalL$. Choose a maximal closed irreducible subvariety $Z \subset S$ containing $s$ which satisfies $G_{\ell,Z}=G_{\ell,s}$.
Then $Z$ is an $\ell$-Galois special subvariety by the very definition.
\end{proof}

\begin{thm}\label{GalLcountable}
The $\ell$-Galois exceptional locus $\ell-GalL$ is a countable union of closed algebraic subvarieties of $S$ defined over $\bar{K}$, and stable under the action of $\mathrm{Gal}(\bar{K}/K)$.
\end{thm}
\begin{proof}
By Proposition \ref{lociunion}, the locus $\ell-GalL$ is the union of all $\ell$-Galois special subvarieties, which are closed algebraic subvarieties defined over $\bar{K}$ by Corollary \ref{fodgalsp}. The same Corollary also shows that the union is stable under the action of the absolute Galois group of $K$. The Theorem follows from the fact that there are only countably many closed algebraic subvarieties of $S$ defined over the countable field $\bar{K}$.
\end{proof}

Proposition \ref{absspecial} shows that $AHL \subset HL$ and $AHL \subset \ell-GalL$. The absolute Hodge conjecture and the absolute Mumford-Tate conjecture imply that these inclusions are in fact equalities.
We prove that this is indeed the case when we restrict to the "positive dimensional" part of these loci and the derived group $G_S^{der}$ is simple.

Denote by $HL_{pos}$ (resp. $AHL_{pos}$, $\ell-GalL_{pos}$) the union of all strict special (resp. strict absolutely special, strict $\ell$-Galois special) subvarieties $Z$ of $S$ which are of positive period dimension.

\begin{thm}\label{loci}
Suppose that $G_S^{der}$ is simple.
Then $$HL_{pos} = AHL_{pos} = \ell-GalL_{pos}.$$
\end{thm}
\begin{proof}
Proposition \ref{absspecial} tells us that $AHL_{pos}$ is contained in both $\ell-GalL_{pos}$ and $HL_{pos}$.
Now let $Z$ be a strict special subvariety that is maximal and of positive period dimension. Then Corollary \ref{application1} implies that $Z$ is in fact absolutely special. We conclude that $HL_{pos} \subset AHL_{pos}$.
Similarly, a maximal strict $\ell$-Galois special subvariety of positive period dimension is absolutely special (Corollary \ref{Corrolary}).
This shows the inclusion $\ell-GalL_{pos} \subset AHL_{pos}$.
\end{proof}

As a consequence, we see that in the situation of Theorem \ref{loci} the $\ell$-Galois exceptional locus $\ell-GalL_{pos}$ of positive period dimension is independent of $\ell$ as predicted by the Mumford-Tate conjecture, but this is not at all clear for the full $\ell$-Galois exceptional locus.
\begin{rem}
Denote by $WS_{pos}$ the union of all strict weakly special subvarieties of $S$ of positive period dimension.
Then one sees that assuming the simpleness of $G_S^{der}$, the above loci are also equal to $WS_{pos}$. The equality $HL_{pos} = WS_{pos}$ is already used crucially in \cite{KO}.
\end{rem}

Theorem \ref{loci} allows us to transfer results on the Hodge locus $HL_{pos}$ of positive period dimension to the $\ell$-Galois exceptional locus of positive period dimension.
For example, by (\cite{KO}, Theorem 1.5) the subset $HL_{pos}$ is either Zariski-dense or a finite union of special subvarieties.
These special subvarieties are maximal, and hence absolutely special by Corollary \ref{application1}.
We obtain that $\ell-GalL_{pos}$ is either Zariski-dense or a finite union of $\ell$-Galois special subvarieties.

The \emph{level} of the variation of Hodge structure $\mathbb{V}$ as defined in (\cite{BKU}, Definition 3.13) is the maximal number $ m $ such that the Hodge structure of weight zero on the Lie algebra of the adjoint generic Mumford-Tate group $G_S^{ad}$ has non-trivial $(-m,m)$-part. In fact, if this Hodge structure is not irreducible then one needs a slightly refined definition that can be found in (\cite{BKU}, Definition 3.12). Baldi-Klingler-Ullmo proved that if $G_S^{der}$ is simple and the level of $\mathbb{V}$ is greater or equal to three, then the Hodge locus $HL_{pos}$ of positive period dimension is a finite union of special subvarieties.
Using this, we obtain the following Corollary to Theorem \ref{loci}:

\begin{cor}\label{ellGalLfiniteunion}
Suppose that $G_S^{der}$ is simple and the level of the variation of Hodge structure $\mathbb{V}$ is greater or equal to three.
Then $\ell-GalL_{pos}$ is a finite union of $\ell$-Galois special subvarieties.
\end{cor}

\subsection{Absolute Mumford-Tate conjecture for projective hypersurfaces}

We show how the results on the $\ell$-Galois exceptional locus can be applied in suitable situations to prove that many members of a family satisfy the (absolute) Mumford-Tate conjecture.

\begin{mydef}
We say that the variation of Hodge structure $\mathbb{V}$ satisfies the \emph{Torelli property} if every closed irreducible subvariety $Z \subset S$ with $H_Z=1$ is a point.
\end{mydef}
\begin{rem}
The Torelli property can be rephrased as saying that the complex period map does not contract a positive dimensional subvariety to a point. In particular, it is satisfied if the period map satisfies an infinitesimal Torelli theorem, since the period map is an immersion in this case.
\end{rem}

\begin{thm}\label{mtmany} \textnormal{ }
\begin{enumerate}[(i)]
\item
Suppose $G_{\ell,S} = G_S \otimes \mathbb{Q}_{\ell}$. Then the Mumford-Tate conjecture holds for a very general complex point of $S$, i.e. away from a countable union of closed algebraic subvarieties.
\item \label{mtdense}
Suppose that $G_{\ell,S} = G_S^{AH} \otimes \mathbb{Q}_{\ell} $, and the derived group $G_S^{der}$ is simple.
We assume furthermore that the level of the variation of Hodge structure $\mathbb{V}$ is at least $3$, and that it satisfies the Torelli property. Then there exists a dense open $K$-subvariety $U \subset S$ such that the absolute Mumford-Tate conjecture holds true for every $x \in U(\mathbb{C}) \setminus U(\bar{K})$.
\end{enumerate}
\end{thm}
\begin{proof}
\begin{enumerate}[(i)]
\item
By the Theorem of Cattani-Deligne-Kaplan (cf. \cite{CDK}) and Theorem \ref{GalLcountable}, both the Hodge locus and the $\ell$-Galois exceptional locus form a countable union of closed algebraic subvarieties of $S$. For any $x \in S(\mathbb{C}) \setminus \left(HL \cup \ell-GalL \right)$ we have the equality
$$G_{\ell,x}=G_{\ell,S} = G_S \otimes \mathbb{Q}_{\ell}=G_x \otimes \mathbb{Q}_{\ell}.$$
\item
Under the assumption that $G_S^{der}$ is simple, it follows from Corollary \ref{ellGalLfiniteunion} that the $\ell$-Galois exceptional locus $\ell-GalL_{pos} $ of positive period dimension is a strict closed algebraic subvariety of $S$ provided that the variation $\mathbb{V}$ has level $\ge 3$. It is moreover defined over $K$ by Corollary \ref{fodgalsp}.
If we let $U \subset S$ be the dense open $K$-subvariety defined as the complement of $\ell-GalL_{pos}$, it follows from the Torelli property that every point $x \in U(\mathbb{C})$ which is not $\ell$-Galois generic in $S$ is an $\ell$-Galois special point, and thus defined over $\bar{K}$ by Corollary \ref{fodgalsp}.
Consequently every $x \in U(\mathbb{C}) \setminus U(\bar{K})$ satisfies
$$G_{\ell,x} = G_{\ell,S} = G_{S}^{AH} \otimes \mathbb{Q}_{\ell}.$$
The inclusions $G_{\ell,x} \subset G_x^{AH} \otimes \mathbb{Q}_{\ell}\subset G_S^{AH} \otimes \mathbb{Q}_{\ell}$ force the equality $G_{\ell,x} = G_x^{AH} \otimes \mathbb{Q}_{\ell}$.
\end{enumerate}
\end{proof}

\begin{rem}\textnormal{ }
\begin{enumerate}[(i)]
\item
More precisely, Theorem \ref{mtmany}(\ref{mtdense}) proves that all points $x \in U(\mathbb{C}) \setminus U(\bar{K})$ are $\ell$-Galois generic in $S$. This it not too surprising since it is conjectured (\cite{BKU}, Conjecture 2.5) that the Hodge locus is an algebraic subvariety of $S$ once the level is greater or equal to three. In particular, the Hodge generic points (and thus conjecturally also the $\ell$-Galois generic points) should form a dense open subset of $S$.
\item
Replacing the first condition in Theorem \ref{mtmany}(\ref{mtdense}) by the condition $G_{\ell,S}=G_S \otimes \mathbb{Q}_{\ell}$, the same argument shows that the Mumford-Tate conjecture holds for all but countably many points $x \in U(\mathbb{C})$.
We can not argue that these countably many points are defined over $\bar{K}$ because of our lack of knowledge about the fields of definition of special points.
\item
Arguing as in (\cite{Moonen}, Theorem 4.3.8), the assumption $G_{\ell,S} = G_S^{AH} \otimes \mathbb{Q}_{\ell}$ (respectively $G_{\ell,S} = G_S \otimes \mathbb{Q}_{\ell}$) is satisfied once the absolute Mumford-Tate conjecture (respectively the Mumford-Tate conjecture) holds for a single point $s \in S(\mathbb{C})$. In other words, the above theorem shows that if these conjectures hold for one point, they hold for many points.
\end{enumerate}
\end{rem}

We give an example to illustrate the scope of the preceding Theorem.
Denote by $\mathcal{M}_{d,n}$ the moduli space of smooth projective hypersurfaces of degree $d$ in $\mathbb{P}^{n+1}$. More precisely, we should add a level structure in order to obtain a fine moduli space defined over $\mathbb{Q}$, but we will ignore this issue here.

\begin{cor}\label{corprojhyp}
Assume that $n \ge 3$, $d \ge 5$, and $(n,d) \not= (4,5)$. There is a dense open $\mathbb{Q}$-subvariety $U \subset \mathcal{M}_{d,n}$ such that the absolute Mumford-Tate conjecture for $H^n_{\mathrm{prim}}$ holds for all $x \in U(\mathbb{C}) \setminus U(\bar{\mathbb{Q}})$.
\end{cor}
\begin{proof}
We show that the assumptions of Theorem \ref{mtmany}(\ref{mtdense}) are satisfied.
We argue as in (\cite{BKU}, 7.2):
As soon as $n \ge 3$, $d \ge 5$ and $(n,d) \not= (4,5)$, the level of the variation of Hodge structure $\mathbb{V}$ is greater or equal to three.
Let $G_{d,n}$ denote the identity component of the automorphism group of the primitive cohomology $H^n_{prim}$ of a smooth projective hypersurface preserving the cup product. A big monodromy result of Beauville (\cite{Beauville}, Theorem 2 and 4) shows that $H_S = G_{d,n}$.
As $G_S^{AH} \subset G_{d,n}$, and $G_S$ and $G_{\ell,S}$ both contain $H_S$, we conclude that $G_S = G_S^{AH}$ and $G_{\ell,S} = G_S^{AH} \otimes_{\mathbb{Q}} \mathbb{Q}_{\ell}$. In addition, $G_S^{der} = G_{d,n}$ is simple.
The Torelli property follows from the fact the period map for $\mathcal{M}_{d,n}$ satisfies the infinitesimal Torelli theorem by a result of Griffiths, see (\cite{Flenner}, Theorem 3.1) for the more general case of complete intersections.
\end{proof}

\subsection{$\ell$-Galois special points and the Mumford-Tate conjecture for abelian varieties}

In this section, we consider the case where $f: \mathcal{X}\to Sh$ is the universal family of abelian varieties over a Shimura variety $Sh$ of Hodge type, and $\mathbb{V} = R^1f^{an}_{*} \underline{\mathbb{Q}}$.
Deligne showed that every Hodge cycle on an abelian variety is absolute Hodge (\cite{Deligne}, Theorem 2.11). Consequently we know that $G_Z = G_Z^{AH}$ for every subvariety $Z \subset Sh$, and every special subvariety of $Sh$ is absolutely special and hence also $\ell$-Galois special.
Moreover, the proof of the Tate conjecture for abelian varieties due to Faltings in \cite{Faltings} shows that $G_{\ell,Z}$ is a reductive group.
We use the existence of CM points to show that Conjecture \ref{everyGalspabssp} ($\ell$-Galois special subvarieties are special) implies the Mumford-Tate conjecture in this case.
We will then prove that one can reduce this conjecture to $\ell$-Galois special subvarieties of dimension zero, that is, to $\ell$-Galois special points.

\begin{thm}\label{mtequivalent}
The following are equivalent:

\begin{enumerate}[(i)]
\item \label{mt}
The Mumford-Tate conjecture \ref{mtconjecture} holds for all (principally polarized) abelian varieties of dimension $g$;
\item \label{CM}
Every $\ell$-Galois special subvariety of $\mathcal{A}_g$ is special;
\item \label{ab}
Every $\ell$-Galois special point $x \in \mathcal{A}_g(\bar{\mathbb{Q}})$ is a CM point.
\end{enumerate}
\end{thm}
\begin{proof}
Clearly, (\ref{mt}) implies all other conditions since it implies that $\ell$-Galois special subvarieties are special.
Let $x \in \mathcal{A}_g(\mathbb{C})$. Take $Z$ to be an $\ell$-Galois special subvariety containing $x$ with $G_{\ell,Z}= G_{\ell, x}$, so that $x$ is $\ell$-Galois generic in $Z$.
Assuming (\ref{CM}), $Z$ is a special subvariety and thus contains a CM point. Using that the Mumford-Tate conjecture is known for CM abelian varieties (\cite{CM}), the result (\cite{Moonen}, Cor. 4.3.15) shows that the $\ell$-Galois generic point $x \in Z(\mathbb{C})$ satisfies the Mumford-Tate conjecture. Hence $(\ref{CM}) \implies (\ref{mt})$.
We now show that $(\ref{ab}) \implies (\ref{CM})$.
Any $\ell$-Galois special subvariety $Z$ is weakly special by Lemma \ref{sabGsws}.
It follows from the description in Example \ref{shimura} that there is a morphism of Shimura varieties $\pi: Sh_{G_Z} \to Sh_{G_Z/H_Z}$ such that $$Z = \pi^{-1}(\{y\})$$ for a point $y \in Sh_{G_Z/H_Z}(\bar{\mathbb{Q}})$.
Denote by $pr: G_{Z} \otimes_{\mathbb{Q}} \mathbb{Q}_{\ell} \to (G_Z/H_{Z}) \otimes_{\mathbb{Q}} \mathbb{Q}_{\ell}$ the projection.
Then $$G_{\ell, Z} = pr^{-1}(G_{\ell,y}).$$
We claim that $y$ is an $\ell$-Galois special point. Suppose there exists a positive dimensional closed irreducible subvariety $Y \subset Sh_{G_Z/H_Z}$ containing $y$ with $G_{\ell, Y} = G_{\ell, y}$. In particular $H_{\ell, Y} \subset G_{\ell, y}$.
Consider the subvariety $W := \pi^{-1}(Y)$. Then $Z \subset W$ and $$H_{\ell, W} = pr^{-1}(H_{\ell,Y}) \subset G_{\ell,Z}.$$
By Lemma \ref{lem}, we obtain the equality $G_{\ell, W} = G_{\ell, Z}$. This contradicts the assumption that $Z$ is $\ell$-Galois special and thus proves that $y$ is an $\ell$-Galois special point.
Now (\ref{ab}) asserts that $y$ is a CM point. We conclude that $Z = \pi^{-1}(\{y\})$ is a special subvariety.
\end{proof}

Note that the formulation of (\ref{ab}) does not require Hodge theory anymore. For example, by Falting's Theorem on endomorphisms of abelian varieties \cite{Faltings}, $x \in \mathcal{A}_g(\bar{\mathbb{Q}})$ is a CM point if and only if $G_{\ell,x}$ is a torus (see also \cite{Moonen}, 2.4.4.).

Let $Sh \subset \mathcal{A}_g$ be a Shimura variety of Hodge type with $G_{Sh}^{der}$ simple.
Corollary \ref{Corrolary} states that part (\ref{CM}) in the above theorem holds true for maximal positive dimensional strict $\ell$-Galois special subvarieties in $Sh$. Due to the existence of CM points we get:

\begin{cor}\label{mGs}
Assume that $G_{Sh}^{der}$ is simple. Let $Z \subset Sh$ be a strict $\ell$-Galois special subvariety which is positive dimensional and maximal for these properties.
Then $G_{\ell, Z} = G_Z \otimes \mathbb{Q}_{\ell}$. Consequently, any $\ell$-Galois generic point of $Z$ satisfies the Mumford-Tate conjecture.
\end{cor}

Even though we cannot prove that all $\ell$-Galois special points are CM points, we know that there are a lot of points which are not $\ell$-Galois special.
For instance, take any positive dimensional subvariety of $\mathcal{A}_g$ and any $\ell$-Galois generic point thereof.
For such points, the Mumford-Tate conjecture holds in the special case that the derived group of the Mumford-Tate group is simple:

\begin{prop}
Let $x \in \mathcal{A}_g(\mathbb{C})$ such that $G_x^{der}$ is simple and such that $x$ is not an $\ell$-Galois special point.
Then $G_{\ell,x}= G_{x} \otimes \mathbb{Q}_{\ell}$, i.e. $x$ satisfies the Mumford-Tate conjecture.
\end{prop}
\begin{proof}
As $x$ is not an $\ell$-Galois special point, there exists a closed irreducible positive dimensional subvariety $Z$ containing $x$ with $G_{\ell,Z}=G_{\ell, x}$. We may assume that $Z$ is weakly special, this does not change $G_{\ell,Z}$ by Lemma \ref{lem}.
Let $Sh_{G_x} \subset \mathcal{A}_g $ be the connected component containing $x$ of the Shimura variety with generic Mumford-Tate group $G_x$.
We have the inclusion $$H_{\ell,Z} \subset G_{\ell,Z} = G_{\ell,x} \subset G_x \otimes \mathbb{Q}_{\ell}.$$
This shows that $H_Z \subset G_x$, and therefore $G_Z = G_x$ by Lemma \ref{lem}. This proves that $Z \subset Sh_{G_x}$.
By the fact that $G_x^{der}$ is simple and $H_Z$ is non-trivial, we get the equality $H_Z = G_x^{der}$. Hence $Z$ is an irreducible component of $Sh_{G_x}$ as $Z$ is weakly special. Using the existence of CM points in special subvarieties of Shimura varieties and (\cite{Moonen}, Cor. 4.3.15) we conclude that $$G_{\ell,x}=G_{\ell, Z}= G_x \otimes \mathbb{Q}_{\ell}.$$
\end{proof}

Together with Corollary \ref{fodgalsp} this shows the following:

\begin{cor}
Let $A$ be a principally polarized complex abelian variety such that $A$ is not defined over a number field. Suppose that the derived Mumford-Tate group $G_A^{der}$ is simple. Then $A$ satisfies the Mumford-Tate conjecture: $G_{\ell,A } =G_{A} \otimes \mathbb{Q}_{\ell}$.
\end{cor}

\section{Local structure of $\ell$-Galois special subvarieties}

We recall how the Theorem of Cattani-Deligne-Kaplan \cite{CDK} can be used to determine the local structure of a special subvariety, and formulate an analogous conjecture for the local structure of an $\ell$-Galois special subvariety (in the case of a family of abelian varieties).
Roughly speaking, this conjecture says that forcing the existence of Tate cycles should cut out an algebraic subvariety of the base.
We then prove that this conjecture implies the Mumford-Tate conjecture for (principally polarized) abelian varieties.

\subsection{Local structure of special subvarieties}

The characterization of special subvarieties as in Definition \ref{defspintro} is a bit ad hoc since we are forcing the algebraicity into the definition: it is not obvious that it is reasonable to look at maximal algebraic subvarieties with a given property if this property is not closely related to algebraicity.

In fact, special subvarieties were initially defined in a different way, and the observation that both definitions coincide relies on a deep result of Cattani-Deligne-Kaplan in \cite{CDK}, as we will now recall.

Let $\Phi: S^{an} \to X:= \Gamma \backslash \mathcal{D}$ be the complex analytic period map attached to the variation of Hodge structure $\mathbb{V}$. The target $X$ is the quotient of a complex analytic space $\mathcal{D}$ which is a connected component of the $G_S(\mathbb{R})$-orbit of a Hodge cocharacter $h: \mathbb{S} \to G_{S, \mathbb{R}}$ by a discrete subgroup $\Gamma \subset G_S(\mathbb{R})$ containing the image of the monodromy representation $\rho: \pi_1(S^{an},s) \to G_S(\mathbb{Q})$.
Here we allow ourselves to replace $S$ by a finite étale cover.

Given a closed irreducible algebraic subvariety $Z \subset S$ with generic Mumford-Tate group $G_Z$, one defines a closed analytic subspace $\mathcal{D}_Z\subset \mathcal{D}$ as a connected component of the $G_Z(\mathbb{R})$-orbit of the Hodge cocharacter $h_x: \mathbb{S} \to G_{Z, \mathbb{R}}$ corresponding to any Hodge generic point $x \in Z(\mathbb{C})$.

If $X_Z$ denotes the image of $\mathcal{D}_Z$ in $\Gamma \backslash \mathcal{D}$, then the period map restricted to $Z$ factors through $X_Z$, that is $$Z \subset \Phi^{-1}(X_Z).$$
Note that $\Phi^{-1}(X_Z)$ is the closed analytic subspace of $S$ defined by forcing the tensors fixed by $G_Z$ to be Hodge tensors.
A fundamental result of Cattani-Deligne-Kaplan shows that this is in fact an algebraic subvariety:

\begin{thm}[\cite{CDK}, see also \cite{BKT}] \label{AlgHodge}
The preimage $\Phi^{-1}(X_Z) \subset S$ is a closed algebraic subvariety of $S$.
\end{thm}
In particular, $Z \subset S$ is special in the sense of Definition \ref{defspintro} if and only if $Z$ is an irreducible component of $\Phi^{-1}(X_Z)$. Thus special subvarieties are subvarieties cut out by certain Hodge tensors.

The concrete description using the period map gives a way of determining the complete local ring $\widehat{\mathcal{O}_{Z,x}}$ at a point $x \in Z(\mathbb{C})$.

After composing $h : \mathbb{S} \to G_{S, \mathbb{R}}$ and $h_x: \mathbb{S} \to G_{Z, \mathbb{R}}$ with
\begin{eqnarray*}
m: \mathbb{G}_{m, \mathbb{C}} & \to & \mathbb{S}_{\mathbb{C}}, \\
t & \mapsto & (t,1)
\end{eqnarray*}
we obtain cocharacters $\mu: \mathbb{G}_{m, \mathbb{C}} \to G_{S, \mathbb{C}}$, and $\mu_x: \mathbb{G}_{m, \mathbb{C}} \to G_{Z, \mathbb{C}}$.
We define the flag varieties $\mathcal{F} := \mathcal{F}(G_S, \mu) = G_{S, \mathbb{C}} / P_{\mu}$ and $\mathcal{F}_Z := \mathcal{F}(G_Z, \mu_x) = G_{Z, \mathbb{C}} / P_{\mu_x}$.
By construction, the Mumford-Tate domains $\mathcal{D} \subset \mathcal{F}$ and $\mathcal{D}_Z \subset \mathcal{F}_Z$ are open analytic subspaces.
The completion of the period map $\Phi$ at $x$ gives rise to a morphism of formal schemes
$$\hat{\Phi}_x: \widehat{S_x} \to \widehat{\mathcal{D}_{y}} = \widehat{\mathcal{F}_{y}}$$ to the completion of the flag variety $\mathcal{F}$ at a point $y \in \mathcal{F}(\mathbb{C})$.
The restriction of the algebraic vector bundle with connection $(\mathcal{V}, \nabla)$
is trivial over $\widehat{S_x}$, cf. (\cite{Katz}, Proposition 8.9), and the period map $\hat{\Phi}_x$ just records the infinitesimal variation of the Hodge filtration on this trivial vector bundle.

The restricted period map $$\hat{\Phi}_x: \widehat{Z_x} \to \widehat{(\mathcal{D}_{Z})_{y}} = \widehat{(\mathcal{F}_{Z})_{y}} \subset \widehat{\mathcal{F}_{y}}$$ factors through the completion of the smaller flag variety $\mathcal{F}_Z$.

The following corollary follows from Theorem \ref{AlgHodge}.
\begin{cor}\label{locsp}
If $Z$ is a special subvariety, the completion $ \widehat{Z_x}$ is an irreducible component of the preimage $\hat{\Phi}_x^{-1}(\widehat{(\mathcal{F}_{Z})_{y}})$
of $\widehat{(\mathcal{F}_{Z})_{y}}$ under the period map.
\end{cor}
\begin{rem}
Here of course an irreducible component of a formal scheme $\operatorname{Spf} R$ is not defined using the underlying topological space, but using the minimal prime ideals of $R$.
\end{rem}

Intuitively, this means the following: the Hodge tensors fixed by $G_{Z}$ give rise to de Rham tensors $t_{\alpha, dR} \in \mathcal{V}_x^{\otimes}$, and the local equations defining $Z$ around the point $x$ are given by the condition that the $t_{\alpha, dR}$ stay in $F^0$, the zeroth piece of the Hodge filtration.

\subsection{Conjecture on the local structure of $\ell$-Galois special subvarieties}

We restrict ourselves to the case of a family of abelian varieties $f: \mathcal{X} \to S$ defined over a number field. It follows from Faltings' work \cite{Faltings} that the groups $G_{\ell,Z}$ are reductive for all $Z \subset S$.
For an $\ell$-Galois special subvariety $Z \subset S$ we want to use Fontaine's $B_{dR}$-comparison in $\ell$-adic Hodge theory to formulate a conjecture describing the complete local ring $\widehat{\mathcal{O}_{Z,x}}$ at a point $x \in Z(\bar{\mathbb{Q}})$.
This conjecture is the analog of Corollary \ref{locsp} for $\ell$-Galois special subvarieties and may thus be viewed as an $\ell$-adic analog of the theorem of Cattani-Deligne-Kaplan on the algebraicity of Hodge loci, in the case of families of abelian varieties.
We first use $\ell$-adic Hodge theory to show that the local period map for an $\ell$-Galois special subvariety $Z$ factors naturally through a flag variety attached to the group $G_{\ell,Z}$, or rather, a de Rham incarnation of this group.

Suppose that $Z \subset S$ is an $\ell$-Galois special subvariety defined over a number field $K$ and let $v$ be a place of $K$ above $\ell$. Let $x \in Z^{sm}(K)$ be a point with a geometric point $\bar{x} \in Z^{sm}(\bar{K})$ lying over $x$.
We denote by $(t_{\alpha,x})_{\alpha} $ the collection of tensors $t_{\alpha,x} \in \mathbb{L}^{\otimes}_{\bar{x}}$ which are fixed by the reductive group $G_{\ell,Z}$. For the local study we may replace $Z$ by a finite étale cover and therefore assume that the tensors $(t_{\alpha,x})_{\alpha}$ extend to global sections $(t_{\alpha})_{\alpha}$ of $\left. \mathbb{L}\right|_{Z_K}^{\otimes}$.
Here we also allow ourselves to replace $K$ by a finite extension.
By (\cite{Scholzerigid}, Theorem 1.10) we have an $\ell$-adic comparison isomorphism of sheaves
$$ \mathbb{L} \otimes_{\mathbb{Q}_{\ell}} \mathcal{O}\mathbb{B}_{dR, S_{v}} \overset{\sim}{\rightarrow} \mathcal{V} \otimes_{\mathcal{O}_{S_{v}}} \mathcal{O}\mathbb{B}_{dR, S_{v}}$$
on the pro-étale site of the rigid analytic space $S^{an}_v$, which is compatible with the connection and the filtration on both sides. Here $S_{v}:= S \otimes_K K_{v}$, and $(\mathcal{O}\mathbb{B}_{dR, S_{v}}, \nabla, Fil^{\bullet})$ is the structural de Rham period sheaf defined in (\cite{Scholzerigid}, Definition 6.8).
After restricting this comparison to the smooth locus $Z_{v}^{sm}$, the tensors $(t_{\alpha})_{\alpha}$ give rise to de Rham tensors $t_{\alpha,\ell-dR}$ in $H^0(Z^{sm}_{v}, \left.\mathcal{V}\right|_{Z^{sm}_{v}}^{\otimes})$ which are horizontal for $\nabla$ and lie in the $F^0$-part of the Hodge filtration.
At the point $x \in Z^{sm}(K_v)$ we thus obtain de Rham tensors $t_{\alpha,\ell-dR,x} \in \mathcal{V}_x^{\otimes}$.
By construction, the $t_{\alpha, \ell-dR,x}$ correspond to the Galois-invariant étale tensors $t_{\alpha,x} \in \mathbb{L}^{\otimes}_{\bar{x}}$ via the $\ell$-adic comparison
$$ \mathcal{V}_x = (\mathbb{L}_{\bar{x}} \otimes_{\mathbb{Q}_{\ell}} B_{dR})^{\mathrm{Gal}(\bar{K}_{v}/K_{v})}.$$
\begin{mydef}\label{definitiongroupGelldR}
We define $G_{\ell-dR, Z}$ to be the $K_{v}$-subgroup of $GL(\mathcal{V}_{x})$ defined by fixing the tensors $t_{\alpha, \ell-dR, x}$.
\end{mydef}
\begin{rem}
We have defined the group $G_{\ell-dR,Z}$ in such a way that the isomorphism $$GL(\mathbb{L}_{\bar{x}}) \otimes_{\mathbb{Q}_{\ell}} B_{dR} \cong GL(\mathcal{V}_x) \otimes_{K_v} B_{dR}$$ induced by Fontaine's $B_{dR}$-comparison restricts to an isomorphism $ G_{\ell,Z} \otimes_{\mathbb{Q}_{\ell}} B_{dR} \cong G_{\ell-dR,Z} \otimes_{K_v} B_{dR}$.
\end{rem}

The group $G_{\ell-dR, Z}$ also has a Tannakian description. Denote by $\rho_x: \mathrm{Gal}(\bar{K}_v / K_v) \to G_{\ell, Z}(\mathbb{Q}_{\ell}) \subset GL(\mathbb{L}_{\bar{x}})$ the Galois representation attached to the point $x \in Z(K_v)$.
Pre-composing a representation of the group $G_{\ell,Z}$ with $\rho_x$ gives a functor $$\operatorname{Rep}_{\mathbb{Q}_{\ell}} G_{\ell,Z} \to \operatorname{Rep}^{dR}_{\mathbb{Q}_{\ell}} \mathrm{Gal}(\bar{K}_v / K_v)$$
to the category of continuous $\ell$-adic Galois representations of $\mathrm{Gal}(\bar{K}_v / K_v)$ which are de Rham in the sense of Fontaine \cite{Fontaine}.
Here we use the fact that every representation of $G_{\ell,Z}$ is isomorphic to a subquotient of some $\mathbb{L}_{\bar{x}}^{\otimes(m,n)}$, and therefore the image under the above functor is a de Rham Galois representation.
Composing with Fontaine's functor \begin{eqnarray*} D_{dR}: \operatorname{Rep}^{dR}_{\mathbb{Q}_{\ell}} \mathrm{Gal}(\bar{K}_v / K_v) & \to & \mathrm{Vect}_{K_v}, \\
(V, \rho) & \mapsto & (V \otimes_{\mathbb{Q}_{\ell}} B_{dR})^{\mathrm{Gal}(\bar{K}_v/K_v)} \end{eqnarray*}
gives a functor $$\mathrm{Rep}_{\mathbb{Q}_{\ell}} G_{\ell,Z} \to \mathrm{Vect}_{K_v}$$
which factors through a $K_v$-linear Tannakian category obtained from $\mathrm{Rep}_{\mathbb{Q}_{\ell}} G_{\ell,Z}$ by tensoring the morphisms with $K_v$ and possibly adding the kernels of new idempotents introduced by this scalar extension.
Then one can interpret the group $G_{\ell-dR,Z}$ as the Tannaka group of this $K_v$-linear Tannakian category with respect to the above fiber functor to $K_v$-vector spaces.

By definition of the group $G_{\ell,Z}$, the monodromy representation $\rho_{\ell, Z_v^{sm}}: \pi_1^{\acute{e}t}(Z^{sm}_v,x) \to GL(\mathbb{L}_{\bar{x}})$ has image in $ G_{\ell,Z}(\mathbb{Q}_{\ell}) $, and thus gives rise to a tensor functor
\begin{equation}\label{functorGellZ} \mathrm{Rep}_{\mathbb{Q}_{\ell}} G_{\ell,Z} \to \mathbb{Q}_{\ell}-\mathrm{LocSys}_{Z_v^{sm}}^{dR} \end{equation}
to the category of $\mathbb{Q}_{\ell}$-local systems on the pro-étale site of $Z_v^{sm}$ which are de Rham
in the sense of (\cite{Scholzerigid}, Definition 8.3). Again, we use that subquotients of the local systems $\left.\mathbb{L}\right|_{Z^{sm}_{v}}^{\otimes(m,n)}$ are de Rham by (\cite{Scholzerigid}, Theorem 8.8(ii)).

There is a relative Fontaine functor
\begin{eqnarray*}D_{dR}: \mathbb{Q}_{\ell}-\mathrm{LocSys}_{Z_v^{sm}}^{dR} & \to & FMIC_{Z_v^{sm,an}}, \\
\mathbb{M} & \mapsto & \nu_*(\mathbb{M} \otimes_{\underline{\mathbb{Q}}_{\ell}} \mathcal{O}\mathbb{B}_{dR, Z_v^{sm}} )
\end{eqnarray*}
to the category of filtered vector bundles with integrable connection on the rigid space $Z_v^{sm,an}$ (cf. \cite{LiuZhu}, 3.2). Here $\nu: Z_{v, pro\acute{e}t}^{sm} \to Z_{v}^{sm,an}$ denotes the natural projection of sites from the pro-étale site to the analytic site.
Composing (\ref{functorGellZ}) with this functor, we obtain a tensor functor
$$ \mathrm{Rep}_{\mathbb{Q}_{\ell}} G_{\ell,Z} \to FMIC_{Z_v^{sm}} $$
which factors through the $K_v$-linear category $\mathrm{Rep}_{K_v} G_{\ell-dR,Z}$,
giving a tensor functor
\begin{equation}\label{Gstructure} \mathrm{Rep}_{K_v} G_{\ell-dR,Z} \to FMIC_{Z_v^{sm}}. \end{equation}

As before, the infinitesimal variation of the Hodge filtration induces a natural period map
$$\hat{\Phi}_{x}: \widehat{S_{v,x}} \to \widehat{\mathcal{F}_{y}} $$
to the completion of a flag variety $\mathcal{F}:= GL(\mathcal{V}_x) / P$ at a point $y \in \mathcal{F}(K_v)$.

By (\cite{Rivano}, IV, 2.4) there exists a cocharacter $\mu_x: \mathbb{G}_{m, K_{v}} \to G_{\ell-dR, Z}$ splitting the Hodge filtration on $\mathcal{V}_{x}$. We denote by $\mathcal{F}_{\ell,Z}:= G_{\ell-dR,Z} / P_{\mu_x}$ the flag variety for the group $G_{\ell-dR, Z}$ and the parabolic $ P_{\mu_x} $ defined by the cocharacter $\mu_x$.

The functor (\ref{Gstructure}) equips the restriction of the filtered vector bundle with connection $(\mathcal{V}, \nabla, F^{\bullet})$ to $Z_v^{sm}$ with a $G_{\ell-dR, Z}$-structure, and hence the restriction of $\hat{\Phi}_x$ to $\widehat{Z^{sm}_{v,x}}$ factors through the completion of the flag variety $\mathcal{F}_{\ell,Z}$:
$$\hat{\Phi}_{x}: \widehat{Z^{sm}_{v,x}} \to \widehat{(\mathcal{F}_{\ell,Z})_{y}} \subset \widehat{\mathcal{F}_{y}}.$$
This shows that the restriction of $\hat{\Phi}_{x}$ to $\widehat{Z_{v,x}}$ factors through $\widehat{(\mathcal{F}_{\ell,Z})_{y}}$.

We are ready to formulate the $\ell$-adic analog of Corollary \ref{locsp} as a conjecture.

\begin{conj}[Local structure of $\ell$-Galois special subvarieties] \label{ellCDK}
If $Z \subset S$ is an $\ell$-Galois special subvariety, $\widehat{Z_{v,x}} $ is an irreducible component of the preimage
$$\hat{\Phi}_x^{-1}(\widehat{(\mathcal{F}_{\ell,Z})_{y}})$$
of $\widehat{(\mathcal{F}_{\ell,Z})_{y}}$ under the period map.
\end{conj}
The conjecture states that an $\ell$-Galois special subvariety $Z$ of $S$ is locally cut out by the condition that the de Rham tensors $t_{\alpha,\ell-dR}$ fixed by $G_{\ell-dR,Z}$ stay in the zeroth piece of the Hodge filtration.
\begin{rem}
It seems difficult to formulate a similar conjecture for more general families, as the group $G_{\ell,Z}$ is not known to be reductive.
\end{rem}

Let us finish with the remark that Conjecture \ref{ellCDK} follows from the Mumford-Tate conjecture.
Indeed, assuming the Mumford-Tate conjecture, every $\ell$-Galois special subvariety $Z$ is special and $G_{\ell,Z} = G_Z \otimes \mathbb{Q}_{\ell}$.
By a result of Blasius (cf. Lemma \ref{groupcomp} below), this shows that $G_{\ell-dR,Z}$ is defined over $K$ and the complex comparison identifies $G_{Z} \otimes_{\mathbb{Q}} \mathbb{C} = G_{\ell-dR,Z} \otimes_K \mathbb{C}$. The desired local structure then follows from Corollary \ref{locsp}.

\subsection{Conjecture on the local structure and the Mumford-Tate conjecture for abelian varieties}

In this section we show that Conjecture \ref{ellCDK} on the local structure of $\ell$-Galois special subvarieties for the Siegel modular varieties $S = \mathcal{A}_g$ implies the Mumford-Tate conjecture for principally polarized abelian varieties.
We hope that this can serve as a motivation for the relevance of studying an analog of the algebraicity theorem of Cattani-Deligne-Kaplan in an $\ell$-adic setting.

Let $K\subset \mathbb{C}$ be a number field and $v$ a place of $K$ dividing $\ell$.
A point $x \in \mathcal{A}_g(K)$ corresponds to a principally polarized abelian variety $A$ over $K$.
We denote by $V_B = H_B^1(A_{\mathbb{C}}, \mathbb{Q})$ the Betti cohomology of $A$, and similarly $V_{dR/K} = H_{dR}^1(A/K)$, $V_{\acute{e}t} = H_{\acute{e}t}^1(A_{\bar{K}}, \mathbb{Q}_{\ell})$.
Let $G_x$ be the Mumford-Tate group of $A$. By Deligne's result (\cite{Deligne}, Theorem 2.11) that Hodge cycles on abelian varieties are absolute Hodge, after possibly replacing $K$ by a finite extension, the subgroup $G_x \otimes \mathbb{C} \subset GL(V_B) \otimes \mathbb{C} \cong GL(V_{dR/\mathbb{C}})$ descents to a $K$-subgroup $ G_{dR, x} \subset GL(V_{dR/K})$.

We need the following Lemma which easily follows from work of Blasius \cite{Blas}.

\begin{lem}\label{groupcomp}
Under Fontaine's $\ell$-adic $B_{dR}$-comparison isomorphism $V_{\acute{e}t} \otimes_{\mathbb{Q}_{\ell}} B_{dR} \cong V_{dR/ K} \otimes_{K} B_{dR}$ the subgroups
$$G_{x, \mathbb{Q}_{\ell}} \otimes_{\mathbb{Q}_{\ell}} B_{dR} = G_{dR,x} \otimes_K B_{dR}$$ are identified.
\end{lem}
\begin{proof}
Being reductive, $G_x$ is the fixator of a collection of Hodge tensors $t_{\alpha} \in V_B^{\otimes}$ (cf. \cite{Deligne}, Proposition 3.1(c)).
The de Rham components $t_{\alpha, dR} \in V_{dR/\mathbb{C}}^{\otimes}$ descent to $K$ (again after allowing a finite extension of $K$) and define the subgroup $G_{dR,x} \subset GL(V_{dR/K})$. A result of Blasius (\cite{Blas}, Theorem 0.3) asserts that the $\ell$-adic comparison isomorphism $V_{\acute{e}t} \otimes_{\mathbb{Q}_{\ell}} B_{dR} \cong V_{dR/K} \otimes_K B_{dR}$ maps the étale components $t_{\alpha, \acute{e}t}$ to $t_{\alpha,dR}$, which proves the claim.
\end{proof}

\begin{lem}\label{splittingMT}
Let $\mu: \mathbb{G}_{m, K_v} \to GL(V_{dR/K_v})$ be any cocharacter splitting the Hodge filtration $F^{\bullet}$ on $V_{dR/K_v}$. Furthermore, let $G \subset GL(V_B)$ be a reductive $\mathbb{Q}$-subgroup such that the group $G \otimes_{\mathbb{Q}} \mathbb{C} \subset GL(V_{dR/\mathbb{C}})$ descents to a $K$-subgroup $G_{dR} \subset GL(V_{dR/K})$.
If $\mu$ factors through $G_{dR} \otimes_K K_v$, then the Mumford-Tate group $G_x$ of $V_B$ is contained in $G$.
\end{lem}
\begin{proof}
By (\cite{Deligne}, Proposition 3.1(c)), the reductive group $G$ is characterized by the tensors $t_{\alpha} \in V_B^{\otimes}$ that it fixes. The corresponding de Rham tensors $t_{\alpha,dR} \in V_{dR/\mathbb{C}}^{\otimes}$ are $\mathbb{C}$-linear combinations of tensors in $V_{dR/K}^{\otimes}$ fixed by the group $G_{dR}$. Since $\mu$ factors through $G_{dR} \otimes_K K_v$, the $t_{\alpha}$ are $\mathbb{Q}$-tensors whose de Rham components lie in $F^{0} \subset V_{dR/\mathbb{C}}^{\otimes}$, and thus are Hodge tensors.
We conclude that the $t_{\alpha}$ are fixed by the Mumford-Tate group of $V_B$, which is thus contained in $G$.
\end{proof}

\begin{thm}\label{ellCDKMT}
Suppose Conjecture \ref{ellCDK} holds true for $S= \mathcal{A}_g$. Then every $\ell$-Galois special point of $\mathcal{A}_g$ is a CM point. In light of Theorem \ref{mtequivalent}, this shows that Conjecture \ref{ellCDK} for this case implies the Mumford-Tate conjecture for principally polarized abelian varieties.
\end{thm}
\begin{proof}
Let $x \in \mathcal{A}_g(K)$ be an $\ell$-Galois special point,
$G_{\ell-dR, x}$ the group defined in Definition \ref{definitiongroupGelldR}, and let $\mu: \mathbb{G}_{m, K_{v}} \to G_{\ell-dR, x}$ be a cocharacter splitting the Hodge filtration on $V_{dR/K_v}$.
In the case of $\mathcal{A}_g$, the local period map is an isomorphism $$\hat{\Phi}_x: \widehat{\mathcal{A}_{g,x}} \overset{\sim}{\rightarrow} \widehat{\mathcal{F}_{y}}$$
which identifies the completion of $\mathcal{A}_g$ at $x$ with the completion of a flag variety attached to the group $GSp_{2g}$ at a point $y$.
Conjecture \ref{ellCDK} predicts that $$\{x\} = \hat{\Phi}_x^{-1}(\widehat{(\mathcal{F}_{\ell,x})_{y}}).$$
As $\hat{\Phi}_x$ is an isomorphism it follows that $\mathrm{dim \,}\mathcal{F}_{\ell,x} = 0$.
This shows that $P_{\mu} = G_{\ell-dR,x}$ and thus the cocharacter $\mu$ factors through the center $Z(G_{\ell-dR,x})$ of $G_{\ell-dR,x}$.
It is known (\cite{Vasiu}, Theorem 1.3.1 or \cite{UllmoYafaev}, Corollary 2.11) that $$Z(G_x)^{\circ} \otimes \mathbb{Q}_{\ell} = Z(G_{\ell,x})^{\circ},$$ i.e. the Mumford-Tate conjecture holds for the identity component of the center.
We thus have an identification
$$Z(G_{dR,x})^{\circ} \otimes_{K} B_{dR} = Z(G_x)^{\circ} \otimes_{\mathbb{Q}} B_{dR} = Z(G_{\ell,x})^{\circ}\otimes_{\mathbb{Q}_{\ell}} B_{dR} = Z(G_{\ell-dR, x})^{\circ} \otimes_{K_v} B_{dR},$$
where we use Lemma \ref{groupcomp} for the first equality.
Hence $Z(G_{dR,x})^{\circ} \otimes_K K_v = Z(G_{\ell-dR,x})^{\circ}$.
In particular, the cocharacter $\mu: \mathbb{G}_{m,K_v} \to GL(V_{dR/K_v})$ factors through $Z(G_{dR,x}) \otimes_K K_v$.
Applying Lemma \ref{splittingMT} to the reductive $\mathbb{Q}$-subgroup $Z(G_x) \subset GL(V_B)$
we conclude that $G_x = Z(G_x)$ is a torus, and thus $x \in \mathcal{A}_g(K)$ is a CM point.
\end{proof}

\addcontentsline{toc}{section}{References}

\bibliographystyle{alpha}
\bibliography{Galsp_arXiv_new}

\end{document}